\documentclass[11pt]{amsart}
\usepackage{mathrsfs}
\usepackage[mathcal]{euscript}
\usepackage{graphicx}
\usepackage{amsthm,amscd}
\usepackage{calc}
\usepackage{mathrsfs,dsfont}
\usepackage{fancyhdr,amscd}
\usepackage{anysize}
\usepackage{amsmath,amssymb,amsfonts}
\usepackage{epsfig,enumerate}
\usepackage{latexsym}
\usepackage{indentfirst, latexsym}
\usepackage{graphics}
\usepackage[all,poly,knot]{xy}
\usepackage[pagebackref]{hyperref}

\providecommand{\U}[1]{\protect\rule{.1in}{.1in}}

\numberwithin{equation}{section}
\def\Re{{\rm Re}}
\def\sup{{\rm Supp}}
\def\Div{{\rm Div}}
\def\Bl{{\rm Bl}}
\def\rk{{\rm rk}}

\def\p{\partial}

\def\b{\bar}

\def\mc{\mathcal}

\theoremstyle{plain}
\newtheorem{thm}{Theorem}[section]
\newtheorem{lemma}[thm]{Lemma}

\newtheorem{prop}[thm]{Proposition}
\newtheorem{cor}[thm]{Corollary}
\newtheorem{quest}[thm]{Question}
\newtheorem{con}[thm]{Conjecture}

\theoremstyle{definition}
\newtheorem{defn}[thm]{Definition}
\newtheorem{rem}[thm]{Remark}
\newtheorem{notation}[thm]{Notation}

\newcommand{\comment}[1]{}

\usepackage{fancyhdr}
\begin{document}

\title[Deformation limit of Moishezon manifolds]{Deformation limit and bimeromorphic embedding of Moishezon manifolds}

\author[Sheng Rao]{Sheng Rao}
\author[I-Hsun Tsai]{I-Hsun Tsai}

\address{Sheng Rao, School of Mathematics and Statistics, Wuhan  University,
Wuhan 430072, People's Republic of China;
Universit\'{e} de Grenoble-Alpes, Institut Fourier (Math\'{e}matiques)
UMR 5582 du C.N.R.S., 100 rue des Maths, 38610 Gi\`{e}res, France}
\email{likeanyone@whu.edu.cn, sheng.rao@univ-grenoble-alpes.fr}

\address{I-Hsun Tsai, Department of Mathematics, National Taiwan University, Taipei 10617,
Taiwan}
\email{ihtsai@math.ntu.edu.tw}

\dedicatory{One world, one fight}
\thanks{Rao is partially supported by NSFC (Grant No. 11671305, 11771339, 11922115) and the Fundamental Research Funds for the Central Universities (Grant No. 2042020kf1065).}
\date{\today}


\begin{abstract}
Let $\pi: \mathcal{X}\rightarrow \Delta$ be a holomorphic family of compact complex manifolds over an open disk in $\mathbb{C}$. If the fiber $\pi^{-1}(t)$ for each nonzero $t$ in an uncountable subset $B$ of $\Delta$ is Moishezon and the reference fiber $X_0$ satisfies the local deformation invariance for Hodge number of type $(0,1)$ or admits a strongly Gauduchon metric introduced by D. Popovici, then $X_0$ is still Moishezon.  We also obtain a bimeromorphic embedding
$\mathcal{X}\dashrightarrow\mathbb{P}^N\times\Delta$.  Our proof can be regarded as a new, algebraic proof of several results in this direction proposed and proved by Popovici in 2009, 2010 and 2013. However, our assumption with $0$ not necessarily being a limit point of $B$ and the bimeromorphic embedding are new. Our strategy of proof lies in constructing a global holomorphic line bundle over the total space of the holomorphic family and studying the bimeromorphic geometry of $\pi:\mathcal{X}\rightarrow \Delta$.  S.-T. Yau's solutions to
certain degenerate Monge--Amp\`ere equations are used.
 \end{abstract}
\maketitle


\section{Introduction} \label{s0}
The deformation limit problem is central in deformation theory, with which
the following longstanding conjecture is concerned. Throughout this paper, one considers the  holomorphic family $\pi: \mathcal{X}\rightarrow \Delta$ of compact complex manifolds of dimension $n$ over an open disk $\Delta$ in $\mathbb{C}$ with the fiber $X_t:=\pi^{-1}(t)$ for each $t\in \Delta$.
\begin{con}\label{conj}
Assume that the fiber $X_t$  is projective for each $t\in \Delta^*:=\Delta\setminus \{0\}$. Then the reference fiber $X_0:=\pi^{-1}(0)$ is Moishezon.
\end{con}
By definition, a compact connected complex manifold $X$ is called a \emph{Moishezon manifold} if it possesses $\dim_{\mathbb{C}}X$ algebraically independent meromorphic functions. Equivalently, $X$ is Moishezon if and only if there exist a projective algebraic manifold $Y$ and a holomorphic modification $Y \rightarrow X$. Any connected projective manifold is Moishezon.

The following is a stronger variant of the above:
\begin{con}\label{conj-M}
  If the fiber $X_t$  is Moishezon for each $t\in \Delta^*$, then the reference fiber $X_0:=\pi^{-1}(0)$ is Moishezon.
\end{con}

The above two conjectures are actually equivalent to:
\begin{con}\label{equiv-conj}
 Let $\pi: \mathcal{X}\rightarrow Y$ be a holomorphic family of compact complex manifolds over a complex variety $Y$, $V\subset Y$ a proper subvariety and write $Y'=Y\setminus V$. Suppose that $X_t$ are Moishezon (or projective) for all $t\in Y'$. Then $X_t$ are Moishezon for all $t\in V$.
\end{con}
In fact, fix a point $t_0$ of $V$, take $D$ as a one-dimensional disc in $Y$ with $t_0$ being the center of $D$ and set $V':=D\cap V$. Then $V'$ is a subvariety of $D$.
Suppose that $D$ is not contained in $V$.  By the identity theorem, $V'$ is a discrete subset of $D$.
Hence by shrinking $D$, we may assume that $V'$ is just the point $t_0$.

D. Popovici proposed proofs of Conjectures \ref{conj}, \ref{conj-M} in \cite{P09,P10}, respectively, and D. Barlet presented several related results to Conjecture \ref{conj-M} in \cite{bar}. 
Our results, some of which have been involved in \cite{P09,P10,P3} for which we propose a new
proof, can be summed up as follows. Recall that for a complex $n$-dimensional manifold $X$, a smooth positive-definite $(1, 1)$-form $\alpha$ on $X$ is said to be a \emph{strongly
Gauduchon metric} if the $(n, n-1)$-form $\partial\alpha^{n-1}$ is $\bar\partial$-exact on $X$. If $X$ carries such a metric, $X$ will be said to be a \emph{strongly Gauduchon
manifold}. This notion was introduced by Popovici in \cite{P3}.

As the main theorem of this paper, we prove more in Theorems \ref{thm-moishezon-update}, \ref{thm-gauduchon-update} and \ref{thm-gauduchon-update'} than the following result:
\begin{thm}[]\label{thm-moishezon}
If the fiber $X_t$  is Moishezon for each nonzero $t$ in an uncountable subset $B$ of $\Delta$,
with $0$ not necessarily being a limit point of $B$, and the reference fiber $X_0$ satisfies the local deformation invariance for Hodge number of type $(0,1)$ or admits a strongly Gauduchon metric, then over
$\Delta_{\epsilon}:=\{z\in \mathbb{C}: |z|<\epsilon\}$ with some small constant $\epsilon>0$,
\begin{enumerate}[$(i)$]
\item
$X_t$ is still Moishezon for any $t\in \Delta_{\epsilon}$.
\item  For some $N\in \mathbb{N}$, there exist
a bimeromorphic map
$$\Phi:\mathcal{X}_{\Delta_{\epsilon}}\dashrightarrow\mathcal{Y}$$
from $\mathcal{X}_{\Delta_{\epsilon}}:=\pi^{-1}(\Delta_{\epsilon})$ to a subvariety $\mathcal{Y}$ of $\mathbb{P}^N\times\Delta_{\epsilon}$
with every fiber $Y_t\subset\mathbb{P}^N\times\{t\}$ being a projective
 variety of dimension $n$, and also a proper analytic set $\Sigma\subset\Delta_{\epsilon}$, such that $\Phi$ induces a bimeromorphic map
 $$\Phi|_{X_t}:X_t\dashrightarrow Y_t$$
 for every
$t\in\Delta_{\epsilon}\setminus \Sigma$.

\end{enumerate}
\end{thm}

In the terminology of \cite[Definition 3.5]{cp}, we may say that our family $\pi:\mathcal{X}\to \Delta$
is \emph{Moishezon}, meaning that it is bimeromorphically equivalent over $\Delta$
to a proper holomorphic map $p:\mathcal{Y}\to \Delta$ from a complex variety, which is $p$-ample, in particular, every fiber $Y_t\subset\mathcal{Y}$ is
a projective variety.   It is remarked in \cite[p. 334]{cp} that not every holomorphic map $f:X\to Y$
between complex spaces such that every fiber $f^{-1}(y)$ is Moishezon, is Moishezon.
Here the meromorphic/bimeromorphic maps are understood and defined in the sense of Remmert (\cite{re}, \cite{st}
and \cite{Ue}).

In contrast to Popovici's approach which is analytic in nature (see Subsection \ref{popovici} for
a brief review), our approach is partly built on
algebraic methods in the sense of Grauert, cf. \cite{bs}.   In the strongly Gauduchon case,
we resort to Monge--Amp\`ere equations of degenerate type with solutions obtained by
S.-T. Yau \cite{Yst} and to Popovici's criterion on big line bundles using mass control \cite{P08}, via
Fujita's approximate Zariski decomposition \cite{fu}.
We use an uncountable subset $B$ (with $0\not\in\bar B$ allowed) for the assumed Moishezon conditions rather than the whole $\Delta^*$ as Popovici does, while the case $0\in\bar B$ is implicit in \cite{P09,P10,P3}, whose approaches are not applicable to our case in Theorem \ref{thm-moishezon} directly, cf. Remark \ref{delta} for example.  To prove Theorem \ref{thm-moishezon}, we obtain an extension property of {\it bigness} in Corollary \ref{unc-big} to the effect that if a global holomorphic line bundle over the total space of the family is such that its restriction to any fiber of an uncountable subset in $\Delta$ is big, then its restriction to any other fiber of the family is also big. Notice that the result \cite[Examples $1$ and $2$]{Tjurin} implies that \lq uncountable' is an indispensable condition there. Moreover,
Campana's counterexample in \cite[Corollary 3.13]{c91} shows that the small deformation of a Moishezon manifold which
is not of general type,  is not necessarily Moishezon. 
Based on these, it is reasonable to propose:
\begin{quest}\label{question}
Characterize those Moishezon manifolds which are still Moishezon after a small deformation.
\end{quest}
\begin{con}[]\label{conj-uncountable}
If the fiber $X_t$  is Moishezon for each $t$ in an uncountable subset of $\Delta$,
then there exists a global holomorphic line bundle $\tilde L$ on $\mathcal{X}$ such
that the restriction $\tilde L|_{X_t}$ is big for every $t\in \Delta$.
\end{con}
Actually, the proof of Theorem \ref{thm-gauduchon-update}.\eqref{thm-gauduchon-update-ii} shows that if Conjecture \ref{conj-uncountable} holds true, then the family $\pi:\mathcal{X}\to \Delta$
is {Moishezon}, which in turn induces a bimeromorphic map on
${X_t}$ for every
$t\in\Delta\setminus\Sigma$ for some proper analytic set $\Sigma\subset\Delta$.

Theorem \ref{thm-moishezon} can be considered as a new understanding of Popovici's remarkable result on deformation limit of projective manifolds from a global and algebraic point of view by a construction of a global holomorphic line bundle over the total space:
\begin{cor}[{\cite[Theorems 1.2, 1.4]{P3}}]\label{thm-pop}
If for each $t\in \Delta^*$, the fiber $X_t$ is projective and the reference fiber $X_0$ satisfies the local deformation invariance for Hodge number of type $(0,1)$ or admits a strongly Gauduchon metric, then $X_0$ is Moishezon.
\end{cor}

The work \cite[Corollary 1.6]{PU} or the $q=1$ case of \cite[Theorem 1.4.(2)]{RZ15} shows that either the \textbf{sGG} condition on $X_0$ or the surjectivity of the natural mapping $\iota^{0,1}_{BC,\bar\partial}$ from the $(0,1)$-Bott--Chern cohomology group of $X_0$ to the Dolbeault one, guarantees that the $(0,1)$-type Hodge numbers of $X_t$ are independent for small $t$.
Notice that by \cite[Remark 3.8]{RZ15} this surjectivity is equivalent to the \textbf{sGG} condition proposed by Popovici--Ugarte \cite{P1,PU};  see also \cite[Theorem 2.1.(iii)]{PU}. Recall that the \emph{\textbf{sGG} condition} for a complex manifold $X$ means that every Gauduchon metric on  $X$ is automatically strongly Gauduchon.

After the completion of this paper, it came to our notice that another work \cite{P19} of Popovici just appeared in which he proposed a new approach to Conjecture \ref{conj-M}. Some results of this paper have already been announced in \cite{rt20}.

\begin{notation}
All compact complex manifolds in this paper are assumed to be connected unless mentioned otherwise,
complex spaces are reduced and the letter $t$ will always denote the parameter for the family of complex manifolds.
The notation $h^i(Z,F)$ denotes $\dim_{\mathbb{C}}H^i(Z,F)$ for a sheaf $F$ of abelian groups over a complex space $Z$.
Sometimes one uses additive notation for the tensor product of vector bundles.
\end{notation}
\noindent
\textbf{Acknowledgement}:
Both authors would like to thank Professor J.-P. Demailly for pointing out the examples in \cite{Tjurin}
which is important for our statement of the main result, and Professor D. Popovici for many useful discussions on many parts of this paper. Part of this work was  completed during the first author's visit to Institute of Mathematics, Academia Sinica from September 2017 to September 2018.
He would like to express his gratitude to the institute for their hospitality and the wonderful work environment during his visit, especially Professors Jih-Hsin Cheng and Chin-Yu Hsiao. We also thank the referee for suggesting certain improvements in several presentations.

\section{Preliminaries}
\subsection{Moishezon manifolds}
To start with, we adopt the following standard terms.  By a \emph{holomorphic family} $\pi:\mc{X}\to B$ of compact complex manifolds, we mean that $\pi$ is a proper holomorphic surjective submersion between complex manifolds as in
\cite[Definition 2.8]{k}.

We are mostly concerned with deformations of Moishezon manifolds.
A nice reference on Moishezon manifolds is \cite[Chapter 2]{mm}.
We first give a geometric description of the Kodaira map. Let $X$ be a compact connected complex manifold of dimension $n$ and $L$ a holomorphic line bundle over $X$. The space of holomorphic sections of $L$ on $X$ is finite-dimensional. Set the linear system associated to $\mathcal{L}=H^{0}(X,L)$ as $|\mathcal{L}|=\{\Div(s): s\in \mathcal{L}\}$. The \emph{base point locus} of the linear system $|\mathcal{L}|$ is given by $$\Bl_{|\mathcal{L}|}=\cap_{s\in \mathcal{L}} \Div(s)=\{x\in X: s(x)=0,\ \text{for all $s\in \mathcal{L}$}\}.$$

Set $d=\dim_{\mathbb{C}} H^{0}(X,L)$ and let $\mathbb{G}(d-1,H^{0}(X,L))$ be the Grassmannian of hyperplanes of $H^{0}(X,L)$. 
  The \emph{Kodaira map} $\Phi_\mathcal{L}$ associated to $L$ is defined by
  $$\Phi_\mathcal{L}: X\setminus \Bl_{|\mathcal{L}|}\longrightarrow \mathbb{G}(d-1,H^{0}(X,L)): x\longmapsto \{s\in H^{0}(X,L): s(x)=0\}.$$

Now consider $\mathcal{L}_p:=H^{0}(X,L^{\otimes p})$ and $\Phi_p:=\Phi_{\mathcal{L}_p}$.
Set $\varrho_p=\max\{\rk\ \Phi_p: x\in X\setminus \Bl_{|\mathcal{L}_p|}\}$ if $\mathcal{L}_p\neq\{0\}$, and $\varrho_p=-\infty$ otherwise. The \emph{Kodaira--Iitaka dimension} of $L$ is
$$\kappa(L)=\max\{\varrho_p:p\in \mathbb{N}^+\}.$$
\begin{thm}[{\cite[Theorem $8.1$]{Ue}}]\label{asymp}
For a Cartier divisor (or a line bundle) $D$ on a variety $M$, there exist positive numbers $\alpha,\beta$ and a positive integer $m_0$ such that for any integer $m\geq m_0$, there hold the inequalities
\begin{equation}\label{asymp-ineq}
 \alpha m^{\kappa(D)}\leq h^0(M, \mathcal{O}_M(mdD))\leq \beta m^{\kappa(D)},
\end{equation}
where $d$ is some positive integer depending on $D$. When the divisor $D$ is effective (or $h^0(M,\mathcal{O}_M(D))\neq 0$), one can take $d=1$ in \eqref{asymp-ineq}.
\end{thm}

A line bundle is called \emph{big} if $\kappa(L)=\dim_{\mathbb{C}} X$. By Siegel's lemma (cf. \cite[Lemma 2.2.6]{mm}) that there exists $C>0$ such that $h^{0}(X,L^{\otimes p})\leq Cp^{\varrho_p}$ for any $p\geq 1$, $L$ is big if (and only if)
\begin{equation}\label{big-crit}
\limsup_{p\rightarrow +\infty}\frac{h^{0}(X,L^{\otimes p})}{p^{n}}>0.
\end{equation}
\begin{lemma}[{\cite[Theorem 2.2.15]{mm}}]\label{big-moi}
A compact complex manifold is Moishezon if and only if it admits a big holomorphic line bundle.
\end{lemma}

\subsection{Locally free sheaves}
Let $f:X\rightarrow Y$ be a continuous map of topological spaces and $\mathcal{F}$ a sheaf of abelian groups on $X$. Denote by $R^qf_*\mathcal{F}$ the \emph{$q$-th direct image sheaf} associated to the presheaf on $Y$
$$V\mapsto H^q(f^{-1}(V),\mathcal{F}),$$
where the restrictions are naturally defined. In particular, $R^0f_*\mathcal{F}$ equals the direct image $f_*\mathcal{F}$.
\begin{lemma}[{cf. \cite[Theorem 6.2 of Chapter III, p. 176]{Iv}}]\label{2-5}
Let $f:X\rightarrow Y$ be a proper map between locally compact spaces and $\mathcal{F}$ a sheaf of abelian groups on $X$. For any point $y\in Y$ and for all $q$,
$$(R^qf_*\mathcal{F})_y\simeq H^q(f^{-1}(y),\mathcal{F}).$$
\end{lemma}

Consider a morphism of ringed spaces
$$f:(X,\mathcal{O}_X)\rightarrow (Y,\mathcal{O}_Y).$$
If $\mathcal{F}$ is an $\mathcal{O}_X$-module, then the sheaves $R^qf_*\mathcal{F}$ admit naturally a structure of $\mathcal{O}_Y$-modules; in particular, if $f$ is a morphism of complex spaces and $\mathcal{F}$ is an analytic sheaf on $X$, then $R^qf_*\mathcal{F}$ are analytic sheaves on $Y$.

\begin{thm}[{Grauert's direct image theorem \cite{Gt} or \cite[\S 2 of Chapter III]{bs}}]\label{gdit}
Let $f:X\rightarrow Y$ be a proper morphism of complex spaces and $\mathcal{F}$ a coherent analytic sheaf on $X$. Then for all $q\geq 0$,
the analytic sheaves $R^qf_*\mathcal{F}$ are coherent.
\end{thm}

An $\mathcal{O}_X$-sheaf $\mathcal{F}$ on a complex space $X$ is called \emph{locally free} at $x\in X$ of rank $p\geq 1$ if there is a neighborhood $U$ of $x$ such that $\mathcal{F}(U)\cong \mathcal{O}_U^p$. Such sheaves are coherent. Using Oka's theorem we get a converse: If a coherent sheaf $\mathcal{F}$ is \emph{free at $x\in X$}, i.e., if the stalk $\mathcal{F}_x$ is isomorphic to $\mathcal{O}_x^p$, then $\mathcal{F}$ is locally free at $x$ of rank $p$. In particular, the set of all points where $\mathcal{F}$ is free is open in $X$.

 For a closer study, one introduces the rank function of an $\mathcal{O}_X$-coherent sheaf $\mathcal{F}$. All $\mathbb{C}$-vector space $\bar{\mathcal{F}}_x:=\mathcal{F}_x/{\mathfrak{m}}_x\mathcal{F}_x$, $x\in X$, are of finite dimension. Here $\mathfrak{m}_x$ is the maximal ideal of $\mathcal{O}_{X,x}$.  The integer $\rk\ \mathcal{F}_x:=\dim_{\mathbb{C}} \bar{\mathcal{F}}_x\in \mathbb{N}$ is called the \emph{rank} of $\mathcal{F}$ at $X$; clearly, $\rk\ \mathcal{O}_x^p=p$. The rank function of a locally free sheaf is locally constant on a complex space $X$. Conversely, if $X$ is reduced and a sheaf  $\mathcal{F}$ is $\mathcal{O}_X$-coherent such that $\rk\ \mathcal{F}_x$ is locally constant on $X$, then $\mathcal{F}$ is a locally free sheaf on $X$.

 The set $S(\mathcal{F})$ of all points in $X$ where a coherent sheaf $\mathcal{F}$ is not free is called the \emph{singular locus} of $\mathcal{F}$. Then:
 \begin{prop}[{\cite[Proposition 7.17]{Rem}}]\label{slia}
 The singular locus $S(\mathcal{S})$, as defined precedingly, of any given $\mathcal{O}_X$-coherent sheaf $\mathcal{S}$ on a complex space $X$ is analytic in $X$. If $X$ is reduced, this set is thin in $X$.
 \end{prop}

Moreover, one has the important:
\begin{thm}[{\cite[Grauert's upper semi-continuity in $\S 10.5.4$]{Grr}}]\label{Upper semi-continuity}
Let $f:X\rightarrow Y$ be a holomorphic family of compact complex manifolds with connected complex manifolds $X,Y$ and $V$ a holomorphic vector bundle on $X$. Then for any integers $i,d\geq 0$, the set $$\{y\in Y: h^i(X_y,V|_{X_y})\geq d\}$$ is an analytic subset of $Y$.
\end{thm}

The topology in $Y$ whose closed sets are all analytic sets is called the \emph{analytic Zariski topology}. The statement of Theorem \ref{Upper semi-continuity} means an upper semi-continuity of $h^i(X_y,V|_{X_y})$ with respect to this analytic Zariski topology.

\section{Deformation limit of projective manifolds: Hodge number}\label{section1}

As a warm-up for the proof of Theorem \ref{thm-moishezon}, we present a weaker version:
Theorem \ref{conj-02} with two proofs given in the following two subsections.
\begin{thm}
\label{conj-02}
	With the additional assumption that $h^{0,2}(X_t)=h^{0,2}(X_0)$ for the $(0,2)$-Hodge numbers as $t$ is close to $0$, Conjecture \ref{conj} holds true.
\end{thm}
Note that Theorem \ref{conj-02} directly follows from \cite[Theorem 1.2]{P3} (or just Corollary \ref{thm-pop} above) since the deformation invariance for the $(0,2)$-Hodge number implies that for the $(0,1)$-Hodge number by Kodaira--Spencer's squeeze \cite[Theorem 13]{KS}. As a corollary of Theorem \ref{conj-02} for the surface case, one obtains:
\begin{cor}[{\cite[Corollary 1.3]{P3}}]\label{decomposition}
 Let $\pi: \mathcal{X}\rightarrow \Delta$ be a holomorphic family of compact complex surfaces such that the fiber $X_t:=\pi^{-1}(t)$ is projective for each $t\in \Delta^*=\Delta\setminus \{0\}$. Then the reference fiber $X_0:=\pi^{-1}(0)$ is also projective.
\end{cor}
\begin{proof}
Here we follow an argument inspired by Popovici \cite{P3}. In fact, the Fr\"{o}licher spectral sequence of any compact complex surface degenerates at $E_1$ as shown in \cite[(2.8) Theorem of Chapter IV]{BHPV} and thus all the Hodge numbers are locally constant for a family of compact complex surfaces as shown in \cite[Proposition 9.20]{V}. See also a power series proof in \cite[Corollary 3.24]{RZ15}. Therefore, the reference surface fiber $X_0$ is Moishezon by Theorem \ref{conj-02}. Moreover, the first Betti number of $X_0$ is even since the Betti numbers of the fibers are always constant and the fiber $X_t:=\pi^{-1}(t)$ is projective for each $t\in \Delta^*$. By Kodaira's classification of surfaces and Siu's result \cite{Sk3} for $K3$ surfaces (or \cite{buch,lam} for a uniform treatment), every compact complex surface with even first Betti number is K\"ahler. Thus, the limit surface $X_0$ is projective since it is both Moishezon and K\"ahler \cite{Moi}.
\end{proof}

We will use the Leray spectral sequence to obtain an isomorphism
$$\Gamma(\Delta, R^i\pi_*\mathcal{O}_{\mathcal{X}})\cong H^i(\mathcal{X},\mathcal{O}_{\mathcal{X}}),\qquad \text{ for any $i\in \mathbb{Z}$}.$$
\begin{thm}[]\label{leray}
Let $X,Y$ be topological spaces, $f:X\rightarrow Y$ a continuous map and $\mathcal{S}$ a sheaf of abelian groups on $X$. Then there exists the \emph{Leray spectral sequence} $(E_r)$ such that
\begin{enumerate}[$(i)$]
    \item \label{}
$E_2^{p,q}\simeq H^p(Y,R^qf_*(\mathcal{S}))$;
    \item \label{}
$(E_r)\Rightarrow H^*(X,\mathcal{S})$.
\end{enumerate}
So $H^k(X,\mathcal{S})=\bigoplus_{p+q=k}E_\infty^{p,q}.$
In particular, if $H^p(Y, R^qf_*(\mathcal{S}))=0$ for $p > 0$, there is an isomorphism
$$H^0(Y,R^qf_*(\mathcal{S}))\cong H^q(X,\mathcal{S}).$$
\end{thm}
\begin{proof}
This just follows from \cite[(13.8) Theorem and (10.12) Special case of Chapter IV]{Dem12}.
\end{proof}

\subsection{Popovici's current-theoretic approach}\label{popovici}
Let us sketch Popovici's approach for Conjecture \ref{conj} to prove Theorem \ref{conj-02}, which indeed inspires us mostly. Most of this subsection is extracted from \cite{P3} and we don't claim any originality here.

It is well-known in deformation theory of complex structures \cite{E,k,MK} that the family of the fibers $X_t$ is $C^{\infty}$-diffeomorphic to a fixed compact differentiable manifold $X$ but equipped with a family of complex structures $J_t$, $t\in \Delta$, varying holomorphically with $t$. In particular, the de Rham cohomology groups $H^k(X_t,\mathbb{C})$ for each $k$ of the fibers are identified with a fixed element of $H^k(X,\mathbb{C})$, while the Dolbeault cohomology groups $H^{p,q}(X_t,\mathbb{C})$ for each $p,q$ may nontrivially depend on $t\in \Delta$.

\begin{lemma}[{\cite[Remark 2.1]{P3}}]\label{leb-neg}
With the setting of Conjecture \ref{conj}, there exists a non-zero integral de Rham cohomology $2$-class $c\in H^2(X,\mathbb{Z})$ such that for every $t\in \Delta^*$, $c$ can be represented by a $2$-form which is of $J_t$-type $(1,1)$.
Moreover, $c$ can be chosen in such a way that for every $t_0\in \Delta\setminus \Sigma$, $c$ is the first Chern class of some ample line bundle $L_{t_0}\rightarrow X_{t_0}$ where $\Sigma=\{0\}\cup \Sigma'\subsetneq \Delta$ and $\Sigma'=\cup \Sigma_\nu$ is a countable union of proper analytic subsets $\Sigma_\nu$ of $\Delta^*$.
\end{lemma}
\begin{proof}
  Consider any class $c\in H^2(X,\mathbb{R})$ and denote by $Z_c$ the set of points $t\in \Delta^*$ such that $c$ can be represented by a $J_t$-type $(1,1)$-form. For every $t\in \Delta^*$,
there holds a Hodge decomposition
$$H^2(X,\mathbb{C})=H^{2,0}(X_t,\mathbb{C})\oplus H^{0,2}(X_t,\mathbb{C})\oplus H^{1,1}(X_t,\mathbb{C})$$
  with Hodge symmetry $H^{2,0}(X_t,\mathbb{C})=\overline{H^{0,2}(X_t,\mathbb{C})}$ since $X_t$ is projective. Thus, $c\in H^2(X,\mathbb{R})$ contains a $J_t$-type $(1,1)$-form if and only if its projection onto $H^{0,2}(X_t,\mathbb{C})$ vanishes.  Since the function
$$\Delta^*\ni t\mapsto\dim H^{0,2}(X_t,\mathbb{C})$$
is locally constant by the projectiveness of $X_t$ for $t\in\Delta^*$ by \cite[Proposition 9.20]{V} or \cite[Theorem 1.4.(2)]{RZ15}, the higher direct image sheaf $R^2\pi_*\mathcal{O}_{\mathcal{X}}$ over $\Delta^*$ is locally free by
 Grauert's continuity theorem \cite[Theorem 4.12.(ii) of Chapter III]{bs} (or just Lemma \ref{gct} below) and hence
can be identified as a holomorphic vector bundle there.
Thus, one sees that $Z_c$ is the zero set of the holomorphic section $s_c\in \Gamma(\Delta^*,R^2\pi_*\mathcal{O}_{\mathcal{X}})$
induced by $c$ (identified with a section of $R^2\pi_*\mathbb{R}$) and followed by $R^2\pi_*\mathbb{R}\rightarrow R^2\pi_*\mathcal{O}_{\mathcal{X}}$.  Thus $Z_c$ is an analytic subset of $\Delta^*$.

By using the projectiveness of $X_t$ for $t\in\Delta^*$, one sees
$$\bigcup_c Z_c=\Delta^*,$$
where the union is countable and taken over all the integral classes $c\in H^2(X,\mathbb{Z})$ with $c$ being the first Chern class of an ample line bundle on some fiber $X_{t_0}, t_0\neq 0$; clearly $t_0\in Z_c$. Notice that a countable union of proper analytic subsets is Lebesgue negligible. So there should be some $c\in H^2(X,\mathbb{Z})$ in the union satisfying $Z_c=\Delta^*$.

To conclude the proof, one uses the standard fact that the ampleness condition is open with respect to the countable analytic Zariski topology of $\Delta^*$, which follows from the Nakai--Moishezon criterion for ampleness and the Barlet theory of cycle spaces.
\end{proof}

Consider a smooth family $\{dV_t\}_{t\in \Delta}$ of smooth (positive) volume forms on $X_t$ normalized by $\int_{X_t}dV_t=1$. For each $t\in \Delta\setminus \Sigma$, applying Yau's theorem \cite{Yst} to the class $c$ in Lemma \ref{leb-neg} viewed as a K\"ahler class on $X_t$, one obtains a smooth $2$-form $\omega_t\in c$, which is a K\"ahler form with respect to the complex structure $J_t$ such that
\begin{equation}
 \label{yau}\omega_t^n(x)=vdV_t,\quad\ x\in X_t,
\end{equation}
where $v=\int_Xc^n>0$.

Let's divide the proof of Theorem \ref{conj-02} into two steps.   The first step is to show that under the deformation invariance of the $(0,2)$-Hodge numbers $h^{0,2}(X_t)$ on all $t\in \Delta$, the family $\{\omega_t\}_{t\in \Delta\setminus \Sigma}$ of K\"ahler forms is uniformly bounded in mass in the sense that, for some constant $C>0$ independent of $t\in \Delta\setminus \Sigma$,
$$0<\int_{X_t}\omega_t\wedge\mathfrak{g}_t^{n-1}\leq C<+\infty,\quad\ \text{for all $t\in \Delta\setminus \Sigma$},$$
where $\{\mathfrak{g}_t\}_{t\in \Delta}$ is a smooth family of Gauduchon metrics after possibly shrinking $\Delta$ about $0$.

In fact, choose $\tilde{\omega}$ as any $d$-closed real $2$-form on $\mathcal{X}$ in the de Rham class $c$.  There exists a smooth real $1$-form $\beta_t$ on $X_t$ such that for every $t\in \Delta\setminus \Sigma$, on $X_t$,
\begin{equation}\label{beta}
\omega_t=i_t^*\tilde{\omega}+d_t\beta_t
\end{equation}
where $i_t: X_t\to \mathcal{X}$ and $d_t$ operates along $X_t$.
Then the mass of $\omega_t$ splits as
\begin{equation}\label{om-mass}
\int_{X_t}\omega_t\wedge\mathfrak{g}_t^{n-1}=\int_{X_t}i_t^*\tilde{\omega}\wedge\mathfrak{g}_t^{n-1}+\int_{X_t}d_t\beta_t\wedge\mathfrak{g}_t^{n-1}.
\end{equation}
The first term in the right-hand side of \eqref{om-mass} is bounded as $t$ varies in a neighborhood of $0$ since $\{\mathfrak{g}_t\}_{t\in \Delta}$ is a smooth family and $i_t^*\tilde{\omega}$ can be viewed as a smooth family 
in a neighborhood of $0$. So one needs only to estimate the second term in the right-hand side of \eqref{om-mass}.

Notice that for every $t\in \Delta\setminus \Sigma$, any solution ${\beta}_t$ of \eqref{beta} can be chosen as the type:
$$\tilde{\beta}_t+\sqrt{-1}\b\p_t\alpha_t,$$
where one explicitly chooses the $(0,1)$-part of the real $1$-form $\tilde{\beta}_t$ as
$$\tilde\beta_t^{0,1}=-\b\p^*_t\mathbb{G}_t\tilde{\omega}_t^{0,2},$$
and $\alpha_t$ as some smooth function on $X$ by the $\partial\bar\partial$-lemma. Here $\tilde{\omega}_t^{0,2}$ is the $(0,2)$-part of $i_t^*\tilde{\omega}$  with respect to the complex structure $J_t$ and $\mathbb{G}_t$ denotes the associated Green's operator to the $\b\p_t$-Laplacian $\square_t$. The verification of the above choice is not difficult and can be made by using the fact that from \eqref{beta}, the $\tilde{\omega}_t^{0,2}$ is a trivial class or equivalently, its harmonic component ($t$-dependent {\it{a priori}}) is zero. Moreover, Stokes' theorem yields
$$\int_{X_t}\sqrt{-1}\p_t\b\p_t\alpha_t\wedge\mathfrak{g}_t^{n-1}=\int_{X_t}\alpha_t\wedge\sqrt{-1}\p_t\b\p_t\mathfrak{g}_t^{n-1}=0,$$
which implies that the mass of $\omega_t$ with respect to $\mathfrak{g}_t^{n-1}$ in \eqref{om-mass} is independent of the choice of $\alpha_t$.
Under the deformation invariance of the $(0,2)$-Hodge numbers $h^{0,2}(X_t)$ for $t\in \Delta$, the family $\{\mathbb{G}_t\}_{t\in \Delta}$ of Green's operators depends smoothly on $t\in \Delta$ by the fundamental result of Kodaira--Spencer \cite[Theorem 7]{KS}.
So the second term in the right-hand side of \eqref{om-mass}
$$\int_{X_t}d\beta_t\wedge\mathfrak{g}_t^{n-1}=-2\Re \int_{X_t}\b\p^*_t\mathbb{G}_t\tilde{\omega}_t^{0,2}\wedge\p_t\mathfrak{g}_t^{n-1}$$
is bounded as $t$ varies in a neighborhood of $0$.   This proves the uniform boundedness in mass, as desired.

We come now to the second step, that is, to produce a closed positive $(1,1)$-current on $X_0$ satisfying the following properties in Theorem \ref{pop-hol} to complete the proof of Theorem \ref{conj-02}:
\begin{thm}[{Rewording of \cite[Theorem 1.3]{P08}}]\label{pop-hol}
Let $X$ be a compact complex $n$-dimensional manifold. If there exists a $d$-closed $(1,1)$-current $T$ on $X$ whose de Rham cohomology class is integral and which satisfies
$$\text{(i) $T\geq 0$;\quad (ii) $\int_X{T_{ac}^n>0}$},$$
then the cohomology class of $T$ contains a K\"ahler current and thus $X$  is Moishezon. Here $T_{ac}$ is the absolutely continuous part of $T$.
\end{thm}

In fact, the uniform mass-boundedness property in the first step yields a weakly convergent subsequence $\omega_{t_k}\rightarrow T$ with $\Delta\setminus \Sigma \ni t_k\rightarrow 0$ as $k\rightarrow +\infty$. The limit current $T\geq 0$ of type $(1,1)$ with respect to the limit complex structure $J_0$ of $X_0$ is $d$-closed and lie in the de Rham class $c$. The semi-continuity property for the top power of the absolutely continuous part in $(1,1)$-currents shows for almost every $x\in X_0$
$$T_{ac}(x)^n\geq \limsup_{k\rightarrow +\infty}\omega_{t_k}(x)^n= v\limsup_{k\rightarrow +\infty}dV_{t_k}(x)=vdV_{0}(x)$$
by \eqref{yau}.
Thus, $$\int_{X_0}{T_{ac}^n}\geq v>0$$
which is just the desired $(ii)$ in Theorem \ref{pop-hol}.

\subsection{Upper semi-continuity approach}\label{usc}
As a second proof for Theorem \ref{conj-02}, we will modify the Lebesgue negligibility argument in Lemma \ref{leb-neg} to get a global holomorphic line bundle on the total space $\mathcal{X}$, and then use Kodaira--Spencer's upper semi-continuity theorem and Demailly's effective ampleness to yield a big line bundle on $X_0$.

The main difference between the following lemma and Lemma \ref{leb-neg} lies in the existence of a
global line bundle on $\mathcal{X}$ proved here.

\begin{lemma}[{}]\label{leb-neg-line}
With the setting of Theorem \ref{conj-02}, there exists a global holomorphic line bundle $L$ on the total space $\mathcal{X}$ such that for every $t\in \Delta\setminus \Sigma$, $L_t:=L|_{X_t}\rightarrow X_t$ is ample, where $\Sigma=\{0\}\cup \Sigma'\subsetneq \Delta$ and $\Sigma'=\cup \Sigma_\nu$ is a countable union of analytic subsets $\Sigma_\nu\subsetneq \Delta^*$.
\end{lemma}
\begin{proof} Although the arguments are for the most part similar to Lemma \ref{leb-neg}, for the sake of clarity
we sketch the main points here.    Consider an ample line bundle $L_{t_0}$ on $X_{t_0}$ with some $t_0\in\Delta^*$ and its first Chern class $c:=c_{1}(L_{t_0})\in H^2(X,\mathbb{Z})$.   This time, we will use the exact sequence
\begin{equation}\label{eshdis}
 \cdots\rightarrow  H^1(\mathcal{X}, \mathcal{O}^*_{\mathcal{X}})\rightarrow H^2(\mathcal{X}, \mathds{Z})
\rightarrow H^2(\mathcal{X}, \mathcal{O}_{\mathcal{X}})\rightarrow\cdots
\end{equation}
obtained by the standard exponential exact sequence
$$0\rightarrow  \mathds{Z}\rightarrow \mathcal{O}_{\mathcal{X}}\rightarrow \mathcal{O}^*_{\mathcal{X}}\rightarrow 0.$$
By the similar argument and notations as in Lemma \ref{leb-neg} using however our assumption
on the deformation invariance
of $h^{0,2}(X_t)$ over the whole $\Delta$, we reach an analytic subset $Z_c$,
the zero set in $\Delta$ of the holomorphic section $s_c\in \Gamma(\Delta,R^2\pi_*\mathcal{O}_{\mathcal{X}})$ induced by $c$.  Now on $\Delta^*$, by the projectiveness of $X_t$ one has
$$\bigcup_c Z_c\supseteq\Delta^*$$
with the union taken over all the integral classes $c\in H^2(X,\mathbb{Z})$ satisfying that $c=c_1(H_t)$ for some ample line bundle $H_t$ on $X_{t}, t\neq 0$.  Since a countable union of proper analytic subsets is Lebesgue negligible, there should be some $\tilde{c}\in \Gamma(\Delta,R^2\pi_*\mathbb{Z})$ induced by some ample line bundle $L_{\tilde{t}_0}$ on $X_{\tilde{t}_0}$ with some $\tilde{t}_0\in\Delta^*$ in the union satisfying $Z_{\tilde{c}}\supseteq\Delta^*$.
This gives $Z_{\tilde{c}}=\Delta$ since $Z_{\tilde{c}}$ is analytic in $\Delta$.  (Here $\tilde{t}_0$ could be different from $t_0$ in the beginning.) That is, $0 \equiv s_{\tilde{c}}\in\Gamma(\Delta, R^2\pi_*\mathcal{O}_{\mathcal{X}})$. By using the identification preceding Theorem \ref{leray} and the long exact sequence above, this
implies that $\tilde c$ is the image of some element in $H^1(\mathcal{X},\mathcal{O}^*_{\mathcal{X}})$. This element is the desired global holomorphic line bundle $L$ on the total space $\mathcal{X}$ because its restriction $L|_{X_{\tilde{t}_0}}$ to $X_{\tilde{t}_0}$ admits the first Chern class $\tilde{c}$ and thus $L|_{X_{\tilde{t}_0}}$ is ample by Nakai--Moishezon criterion.

The remaining reasoning can be repeated as in the last paragraph of Lemma \ref{leb-neg}.
\end{proof}

\begin{rem}\label{get-L0}
Notice that the restriction $L|_{X_{\tilde{t}_0}}$ to $X_{\tilde{t}_0}$ of the global holomorphic line bundle $L$ on the total space $\mathcal{X}$ is not necessarily equal to $L_{\tilde{t}_0}$ (in the middle of the proof) although they have the same $c_1$.  Nevertheless, by the argument in \cite[Lemma 2.1]{Weh}, one can construct a new global holomorphic line bundle $L'$ on the total space $\mathcal{X}$ such that its restriction $L'|_{X_{\tilde{t}_0}}$ to $X_{\tilde{t}_0}$ is just $L_{\tilde{t}_0}$ by using the commutative diagram of long exact
sequences
\begin{equation}\label{co-lo-se2}
\xymatrix@C=0.5cm{
  \cdots \ar[r]^{}
  & H^{1}(\mathcal{X},\mathcal{O}_{\mathcal{X}}) \ar[d]_{} \ar[r]^{}
  & H^{1}(\mathcal{X},\mathcal{O}_{\mathcal{X}}^*)\ar[d]_{} \ar[r]^{}
  & H^{2}(\mathcal{X},\mathbb{Z}) \ar[d]_{} \ar[r]^{}
  & H^{2}(\mathcal{X},\mathcal{O}_{\mathcal{X}}) \ar[d]_{} \ar[r]^{} & \cdots \\
   \cdots \ar[r]
  & H^{1}({X_{\tilde{t}_0}},\mathcal{O}_{{X_{\tilde{t}_0}}})\ar[r]^{}
  &H^{1}({X_{\tilde{t}_0}},\mathcal{O}_{X_{\tilde{t}_0}}^*) \ar[r]^{}
  & H^{2}({X_{\tilde{t}_0}},\mathbb{Z}) \ar[r]
  & H^{2}({X_{\tilde{t}_0}},\mathcal{O}_{{X_{\tilde{t}_0}}})\ar[r]&\cdots.}
\end{equation}
Note that the condition on the deformation invariance of $h^{0,1}$ as required in \cite{Weh} is implied by
that of ours on $h^{0,2}$ (see the remarks after Theorem \ref{conj-02}).
 \end{rem}

We are ready to finish our second proof after invoking the following two theorems.

\begin{thm}[Asymptotic Riemann--Roch, {\cite[Corollary 1.1.25]{laz}}]\label{arr}
Let $E$ be a holomorphic vector bundle and $L$ a holomorphic line bundle on a compact complex manifold $X$ of dimension $n$. If $H^i(X,E\otimes L^{\otimes m})=0$ for $i>0$ and $m\gg 0$, then
\begin{equation}\label{rrf}
  h^0(X,E\otimes L^{\otimes m})=\rk\ E\cdot \frac{\int_X c_1(L)^n}{n!}m^n+o(m^n)
\end{equation}
for large $m$. More generally, \eqref{rrf} holds provided that
$h^i(X,E\otimes L^{\otimes m})=o(m^n)$
for $i>0$.
\end{thm}

The other result is on the effective very ampleness.
Here and henceforth, let $K_M$ denote the canonical line bundle of the complex manifold $M$.
\begin{thm}[{\cite[Corollary 2]{Dem93} and also \cite{s96}}]\label{Dem-fuj}
If $L$ is an ample line bundle over an $n$-dimensional projective manifold $X$, then $K_X^{\otimes 2}\otimes L^{\otimes k}$ is very ample for $k>2C(n):=4C_nn^n$ with $C_n<3$ depending only on $n$.
\end{thm}

It is obvious that for some fixed $k>C(n)$, Theorem \ref{Dem-fuj} implies that $K_X\otimes L^{\otimes k}$ is ample (cf. also \cite{fu87}) and thus Kodaira vanishing theorem gives
$$H^q(X,(K_X\otimes L^{\otimes k})^{\otimes m})=0,\ \text{for $q,m\geq 1$}.$$
So asymptotic Riemann--Roch Theorem \ref{arr} gives
$$h^0(X,(K_X\otimes L^{\otimes k})^{\otimes m})=\frac{\int_X c_1(K_X\otimes L^{\otimes k})^n}{n!}m^n+o(m^n),\ \text{for $q,m\geq 1$},$$
where
$$\int_X c_1(K_X\otimes L^{\otimes k})^n>0.$$
Back to the proof of Theorem \ref{conj-02}, let's recall the line bundle $L$ on $\mathcal{X}$ constructed in Lemma \ref{leb-neg-line}. By Kodaira--Spencer's upper semi-continuity \cite[Theorem 4]{KS}, one has
$$h^0(X_0,(K_{X_0}\otimes L_0^{\otimes k})^{\otimes m})\geq h^0(X_t,(K_{X_t}\otimes L_t^{\otimes k})^{\otimes m})$$
for any small $t\in \Delta$ in a neighbourhood of $0$ and thus
\begin{equation}\label{ample-h0}
\limsup_{m\rightarrow +\infty}\frac{h^0(X_0,(K_{X_0}\otimes L_0^{\otimes k})^{\otimes m})}{m^{n}}\geq \frac{\int_{X_t} c_1(K_{X_t}\otimes L_t^{\otimes k})^n}{n!}>0
\end{equation}
for any small $t\in \Delta\setminus \Sigma$ in a neighbourhood of $0$ since then $L_t$ is ample. Thus, by \eqref{big-crit}, $K_{X_0}\otimes L_0^{\otimes k}$ is a big line bundle on $X_0$ and so $X_0$ is Moishezon.  This completes
our second proof of Theorem \ref{conj-02}.

\begin{rem}
Demailly's effective very ampleness in Theorem \ref{Dem-fuj} is crucial in this proof.  For, we need a $t$-independent bound for $k$, with $K_{X_t}\otimes L_t^{\otimes k}$ ample, such that in the upper semi-continuity as $m\rightarrow +\infty$ for which $t$ possibly becomes smaller, the first inequality in \eqref{ample-h0} can still hold.
\end{rem}

\section{Proof of Main Theorem \ref{thm-moishezon}}
The basic idea is to construct a global holomorphic line bundle over the total space by use of torsion and techniques
of currents, in Proposition \ref{global-B} and Theorem \ref{thm-gauduchon-update}, respectively, and then use the extension of bigness done in Corollary \ref{unc-big} to conclude that the restriction of the constructed global holomorphic line bundle to any fiber is big.
\subsection{Deformation density of Kodaira--Iitaka dimension}
We first describe the deformation behavior of Kodaira--Iitaka dimension. Throughout this subsection, we consider the holomorphic family $\pi: \mathcal{X}\rightarrow  Y$ of  compact complex $n$-dimensional manifolds over a connected complex manifold $ Y$ of dimension one with $X_t:=\pi^{-1}(t)$ for $t\in Y$.
\begin{prop}\label{ki-dim-limit}
Assume that there exists a holomorphic line bundle $L$ on $\mathcal{X}$  and set $L_t:=L|_{X_t}$. If the Kodaira--Iitaka dimension $\kappa(L_t)=\kappa$ for each $t$ in an uncountable set $B$ of $ Y$, then $\kappa(L_t)\geq \kappa$ for all $t\in  Y$.
\end{prop}

\begin{proof}
  The case $\kappa=-\infty$ is trivial and so one assumes that $\kappa\geq 0$.
For any two positive integers $p,q$, set
$$\label{tpql}
T_{p,q}(L)=\{t\in  Y: h^0(X_t,L_t^{\otimes p})\geq \frac{1}{q}p^\kappa\}.
$$
By the upper semi-continuity Theorem \ref{Upper semi-continuity}, it is known that $T_{p,q}(L)$ is a proper analytic subset of $Y$ or equal to $Y$. From the assumption that $\kappa(L_t)=\kappa$ for each $t\in B$ and Theorem \ref{asymp}, it holds that
$$\bigcup_{p,q\in \mathbb{N}^+}T_{p,q}(L)\supseteq B$$
despite that the $d$ in Theorem \ref{asymp} is not necessarily one here.
Then since a proper analytic subset of a one-dimensional manifold is at most countable, there exists some analytic subset in the union, denoted by $T_{P,Q}(L)$, such that $T_{P,Q}(L)\supseteq B$ and so $T_{P,Q}(L)= Y$. That is, for any $t\in  Y$, there holds
$$h^0(X_t,L_t^{\otimes P})\geq \frac{1}{Q}P^\kappa.$$

Now take $H=L^{\otimes P}$ and thus for any $t\in Y$,
$$h^0(X_t,H_t)\neq 0.$$ For any two positive integers $p,q$, we write
$$S_{p,q}=T_{p,q}(H)\cap T_{p+1,q}(H)\cap T_{p+2,q}(H)\cap\cdots,$$
where
\begin{equation}\label{tpqh}
T_{p,q}(H)=\{t\in  Y: h^0(X_t,H_t^{\otimes p})\geq \frac{1}{q}p^\kappa\}
\end{equation}
and similarly for $T_{p+1, q}(H), T_{p+2, q}(H),\cdots$.  By the upper semi-continuity again $S_{p,q}$ is a proper analytic subset of $ Y$ or equal to $ Y$ (\cite[(5.5) Theorem of Chapter II]{Dem12}). From the assumption that $\kappa(H_t)=\kappa(L_t)=\kappa$ for each $t\in B$ and Theorem \ref{asymp} with $d=1$, one sees that
$$\bigcup_{p,q\in \mathbb{N}^+}S_{p,q}\supseteq B.$$
Here a proper analytic subset is at most countable, hence there exists some analytic subset in the union, denoted by $S_{I,J}$, such that $S_{I,J}\supseteq B$ and so $S_{I,J}= Y$. That is, for any $t\in  Y$, there holds for all $i\geq I$,
\begin{equation}\label{16-3}
h^0(X_t,H_t^{\otimes i})\geq \frac{1}{J}i^\kappa.
\end{equation}
Hence,
$\kappa(L_t)=\kappa(H_t)\geq \kappa$ for all $t\in  Y$ by Theorem \ref{asymp}  again (with $d=1$).
\end{proof}

\begin{rem}
Our inequality \eqref{16-3} proves more than what is asserted in the proposition. Indeed, \eqref{16-3} plays crucial roles
in many places of this paper, particularly in
Remark \ref{delta}. For instance, \eqref{16-3} is used to obtain Lemma \ref{global-lemma2}.\ref{global-lemma2-ii})
while Lemma \ref{global-lemma2}.\ref{global-lemma2-ii}) is applied to the beginning of the
proof of Theorem \ref{thm-gauduchon-update}.

Moreover, Proposition \ref{ki-dim-limit} is equivalent to Lieberman--Sernesi's main result \cite[Theorem on p. $77$]{ls} in some sense:
With the notations in  Proposition \ref{ki-dim-limit}, there exist a constant $\ell$ and
a set $W \subseteq  Y$, which is the complement of the union of a countable number of proper closed subvarieties, such that
\[ \begin{cases}
\kappa(L_t)=\ell,\ &\text{if $t\in W$;}\\[4pt]
\kappa(L_t)>\ell, &\text{if $t\in  Y\setminus W$.}
\end{cases} \]

We first prove Proposition \ref{ki-dim-limit} by \cite[Theorem on p. $77$]{ls}. If not, there exists some $L_t$, whose Kodaira--Iitaka dimension $\kappa(L_t)< \kappa$.  Then, by their notations, the $B$ in Proposition \ref{ki-dim-limit} must lie outside $W$.  Now, their theorem tells us that $ Y\setminus W$ is only a countable union of analytic subsets of
$ Y$, which is again a countable set, denoted by $E$.   Since $E$ contains $B$ as just mentioned
and $B$ is assumed to be uncountable, a contradiction follows.

Next we prove the converse. One takes $\ell$ as $\min_{t\in  Y} \kappa(L_t)$ and writes $\ell=-1$ for the moment if $\ell=-\infty$.  First suppose that $Y$ is of dimension one.  Define
$T_{p, q}(L)$ similarly (by setting $\kappa$ there to be $\ell$ here) and use $Y$ in place of $B$ to reach $H$.  So the set
$$\{t\in  Y: \kappa(L_t)\geq \ell+1\}=\bigcup_{p,q\in \mathbb{N}^+}S_{p,q},$$
where $S_{p,q}=S_{p,q}(H)$ is the set according to \eqref{tpqh} with $\kappa$ replaced by $\ell+1$, is just the desired countable union of proper closed subvarieties in \cite[Theorem on p. $77$]{ls}.   Note that if we avoid $H$ and define $S_{p,q}=S_{p,q}(L)$, then the above equality becomes only a reverse inclusion ($\supset$) in view of Theorem \ref{asymp}.  Note also that this part of argument remains valid and applicable to $Y$ of higher dimensions.  The converse is proved.
 \qed
\end{rem}

As a direct application of Proposition \ref{ki-dim-limit}, one has the extension property of bigness:
\begin{cor}\label{unc-big}
Let $\pi: \mathcal{X}\rightarrow  Y$ be a holomorphic family of compact complex $n$-dimensional manifolds. Assume that there exists a holomorphic line bundle $L$ on $\mathcal{X}$ such that for each $t$ in an uncountable set of $ Y$, $L|_{X_t}$ is big. Then for each $t\in  Y$, $L|_{X_t}$ is also big and thus $X_t$ is Moishezon.
\end{cor}

\begin{rem} Using the above argument, one is able to partially solve the following
conjecture \cite[Conjecture in Remark 7.6]{Ue}:
The Kodaira dimension is upper semi-continuous under small deformations of complex manifolds.
In fact, one assumes that $\kappa(X_{t_0})=\kappa_0\geq 0$ for some $t_0\in  Y$. Then there exists some positive integer $p_0$ such that $$h^0(X_{t_0},K_{X_{t_0}}^{\otimes p_0})\neq 0.$$
 Without loss of generality, one assumes $p_0=1$ since $\kappa(L^{\otimes q})=\kappa(L)$ for a line bundle $L$ and any positive integer $q$. So there exists some $\beta>0$ such that for all sufficiently large integer $p$,
$$h^0(X_{t_0},K_{X_{t_0}}^{\otimes p})< \beta p^{\kappa_0}.$$
For the $p$ as above, set
$$V_{p}=\{t\in Y: h^0(X_t,K_{X_t}^{\otimes p})\geq \beta p^{\kappa_0}\},$$
which is obviously a proper analytic subset of $ Y$. Let $V= Y\setminus \bigcup_p V_p$. Then for any $t\in V$, there holds
$$h^0(X_t,K_{X_t}^{\otimes p})< \beta p^{\kappa_0},$$
which implies $\kappa_t\leq \kappa_0$.
For global line bundles other than canonical line bundles, the similar argument and conclusion hold.
\qed
\end{rem}

\subsection{Existence of line bundle over total space: Hodge number} \label{ext-lb-h} 
First we introduce the torsion sheaf as in \cite[\S 7.5]{Rem}.
Let $X$ be a reduced complex space. For every $\mathcal{O}_X$-module $\mathcal{F}$ one defines
$$T(\mathcal{F}):=\bigcup_{x\in X}T(\mathcal{F}_x),$$
where $$T(\mathcal{F}_x):=\{s_x\in \mathcal{F}_x: g_xs_x=0,\ \text{for a suitable $g_x\in \mathcal{A}_x$}\}$$
with the multiplicative stalk $\mathcal{A}_x$ of the subsheaf $\mathcal{A}$ of all non-zero divisors in $\mathcal{O}_X$.
For our cases below we have $\mathcal{A}=\mathcal{O}_X$.  We obtain an  $\mathcal{O}_X$-module in $\mathcal{F}$; obviously $T(\mathcal{F})_x=T(\mathcal{F}_x)$. $T(\mathcal{F})$ is called the \emph{torsion sheaf} of $\mathcal{F}$. We call $\mathcal{F}$ \emph{torsion free at $x$} if $T(\mathcal{F})_x=0$. The sheaf $\mathcal{F}/T(\mathcal{F})$ is torsion free everywhere. Subsheaves of locally free sheaves are torsion free.

 \begin{prop}[]\label{s-torsion}
For a proper morphism $\pi: \mathcal{X}\rightarrow Y$ of compact complex manifolds over a connected one-dimensional manifold with $0\in Y$,
suppose that $R^2\pi_*\mathcal{O}_{\mathcal{X}}$ is locally free on $Y^*:=Y\setminus \{0\}$.  Let $s\in \Gamma( Y, R^2\pi_*\mathcal{O}_{\mathcal{X}})$.  Then $s|_{ Y^*}=0$ is equivalent to the germ $s_0\in T(R^2\pi_*\mathcal{O}_{\mathcal{X}})_0$.
As a consequence, if $T(R^2\pi_*\mathcal{O}_{\mathcal{X}})_0=0$ and $s|_{ Y^*}=0$, then $s=0$ in $\Gamma( Y, R^2\pi_*\mathcal{O}_{\mathcal{X}})$.
\end{prop}
\begin{proof} In the algebraic category this is standard, cf. \cite[Lemma 5.3 of Chapter II]{Ht}. Here we work in the
analytic category.  First, $0=s|_{Y^*}\in  \Gamma( Y^*, R^2\pi_*\mathcal{O}_{\mathcal{X}})$ implies that $\sup(s)\cap  Y^*=\emptyset$ and thus $\sup(s)\subseteq \{0\}$.
We need \cite[Proposition 3.2 of Chapter II]{bs}:
Let $Z$ be a complex space, $A$ a closed analytic subset of $Z$ and $\mathcal{F}$ a coherent analytic sheaf on $Z$. Then the sheaf $\mathcal{H}_A^0\mathcal{F}$ is coherent and is equal to the subsheaf of all sections of $\mathcal{F}$ annihilated by suitable powers of the ideal $\mathcal I(A)$. Here the subsheaf $\mathcal{H}_A^0\mathcal{F}$ of $\mathcal{F}$ is formed by the sections whose support is in $A$.
Accordingly, $s$ is a section of $\mathcal{H}_{\{0\}}^0 R^2\pi_*\mathcal{O}_{\mathcal{X}}$ and
since $t$ generates $\mathcal{I}(A)$, there exists some $m>0$ such that, after possibly shrinking $Y$,
$$t^m\cdot s=0\in \Gamma( Y, R^2\pi_*\mathcal{O}_{\mathcal{X}})$$
which implies $s_0\in T(R^2\pi_*\mathcal{O}_{\mathcal{X}})_0$.

Conversely, if $s_0\in T(R^2\pi_*\mathcal{O}_{\mathcal{X}})_0$, then it follows that there exists some $m>0$ such that
$$t^m\cdot s=0\in \Gamma( Y, R^2\pi_*\mathcal{O}_{\mathcal{X}}).$$
By the local freeness of $R^2\pi_*\mathcal{O}_{\mathcal{X}}$ and $t^m\neq 0$ both on $ Y^*$, we obtain $0=s|_{Y^*}\in  \Gamma( Y^*, R^2\pi_*\mathcal{O}_{\mathcal{X}})$, as to be proved. Indeed, the germs $s_t$ lies in the stalk ${(R^2\pi_*\mathcal{O}_\mathcal{X})}_t$ and if $t$ is nonzero, ${(R^2\pi_*\mathcal{O}_\mathcal{X})}_t$ is a free $\mathcal{O}_t$-module by the local freeness of $R^2\pi_*\mathcal{O}_\mathcal{X}$ over $ Y^*$. Since $\mathcal{O}_t$ has no zero-divisor, $t^m \cdot s=0$ on $ Y^*$ yields the vanishing of $s_t$ for all $t\in  Y^*$.    For the last statement, it follows from the first part that the germs $s_y=0$ for all $y\in Y$, giving $s\equiv 0$.
\end{proof}

Here we need a simple result from commutative algebra:
 \begin{lemma}[{\cite[Exercise 2 on p. 31]{am}}]\label{caresult} Let $\mathcal{A}$ be a commutative ring with $1$, $\mathfrak{a}$ an ideal of $\mathcal{A}$ and $\mathfrak{M}$ an $\mathcal{A}$-module. Then $\mathfrak{M} \otimes_\mathcal{A} \mathcal{A}/\mathfrak{a}$ is isomorphic to $\mathfrak{M}/\mathfrak{a}\mathfrak{M}$.
 \end{lemma}

We also collect several (partial) results on Grauert's base change theorem. In this part one always considers the proper morphism $f:X\rightarrow Y$ of complex spaces and a coherent analytic sheaf $\mathcal{F}$ on $X$, \emph{flat with respect to $f$}, which means that the $\mathcal{O}_{f(x)}$-modules $\mathcal{F}_x$ are flat for all $x\in X$.
 \begin{lemma}[Grauert's base change theorem, {\cite[Theorem 3.4 of Chapter III]{bs}}]\label{gbc}  Let $y$ be a point in $Y$, and $q$ an integer. The following assertions are equivalent:
\begin{enumerate}[$(i)$]
    \item \label{gbci}
The canonical morphisms
$$R^qf_*(\mathcal{F})_y\otimes_{\mathcal{O}_y} M\rightarrow R^qf_*(\mathcal{F}\otimes_{\mathcal{O}_X} f^*M)_y$$
 are isomorphisms, where $M$ is an arbitrary  $\mathcal{O}_y$-module of finite type.
    \item \label{}
The functor $M\mapsto R^qf_*(\mathcal{F}\otimes_{\mathcal{O}_X} f^*M)_y$ is right exact.
    \item \label{}
The functor $M\mapsto R^{q+1}f_*(\mathcal{F}\otimes_{\mathcal{O}_X} f^*M)_y$ is left exact.
    \item \label{gbciv}
The canonical map $$R^qf_*(\mathcal{F})_y\rightarrow R^qf_*(\mathcal{F}/\hat{\mathfrak{m}}_y\mathcal{F})_y$$ is surjective, where $\mathfrak{m}_y$ is the maximal ideal of $\mathcal{O}_{Y,y}$ or the natural ideal-sheaf given by this and $\hat{\mathfrak{m}}_y$ is the ideal-sheaf of $\mathcal{O}_X$ generated by the inverse image of $\mathfrak{m}_y$.
 \end{enumerate}
 \end{lemma}
As a corollary of Lemma \ref{gbc}, one gets the exactness criterion:
\begin{cor}[{\cite[Corollary 3.7 of Chapter III]{bs}}]\label{exact-cri}
With the assumptions of Lemma \ref{gbc}, the following assertions are equivalent:
\begin{enumerate}[$(a)$]
    \item \label{}
The functor $M\mapsto R^qf_*(\mathcal{F}\otimes_{\mathcal{O}_X} f^*M)_y$ is exact.
   \item \label{}
The canonical maps
$$R^qf_*(\mathcal{F})_y\rightarrow R^qf_*(\mathcal{F}/\hat{\mathfrak{m}}_y\mathcal{F})_y,\ R^{q-1}f_*(\mathcal{F})_y\rightarrow R^{q-1}f_*(\mathcal{F}/\hat{\mathfrak{m}}_y\mathcal{F})_y$$
are surjective.
 \end{enumerate}
\end{cor}
One says that \emph{$\mathcal{F}$ is cohomologically flat in dimension $q$ at the point $y$} if the equivalent conditions of the criterion in Corollary \ref{exact-cri} are fulfilled;
 $\mathcal{F}$ is \emph{cohomologically flat in dimension $q$ over $Y$} if $\mathcal{F}$ is cohomologically flat  in dimension $q$ at any point of $Y$.
\begin{lemma}[Grauert's continuity theorem, {\cite[Theorem 4.12.(ii) of Chapter III]{bs}}]\label{gct}
With the assumptions of Lemma \ref{gbc}, if $\mathcal{F}$ is cohomologically flat in dimension $q$ over $Y$ in the sense
as given precedingly, then the function
$$y\mapsto \dim H^q(X_y, \mathcal{F}_y)$$
is locally constant. Conversely, if this function is locally constant and $Y$ is a reduced space, then $\mathcal{F}$ is cohomologically flat in dimension $q$ over $Y$; in particular, the sheaf $R^qf_*(\mathcal{F})$ is locally free.
\end{lemma}
Based on the above, one obtains:
\begin{cor}[{\cite[Theorem (8.5).(iv) of Chapter I]{BHPV}}]\label{ccs-bc}
Let $X$, $Y$ be reduced complex spaces and $f : X\rightarrow Y$ a
proper holomorphic map. If $\mathcal{F}$ is any coherent sheaf on $X$, which is flat with respect to $f$, and $h^q(X_y,\mathcal{F}_y)$
is constant in $y\in Y$, then the base-change map  $$R^qf_*(\mathcal{F})_y/{\mathfrak{m}_yR^qf_*(\mathcal{F})_y}\rightarrow R^qf_*(\mathcal{F}/\hat{\mathfrak{m}}_y\mathcal{F})_y$$ is bijective.
\end{cor}
\begin{proof}
 Use Lemmas \ref{gct} and \ref{gbc} or Lemma \ref{gct} and Corollary \ref{exact-cri}.
\end{proof}

\begin{rem} Note that $H^q(X_y, (\mathcal{O}_{\mathcal{X}})_y) =H^q(X_y, \mathcal{O}_{X_y})$.
For, $\mathcal{F}_y = i_y^*\mathcal{F}$, where $i_y: X_y \to \mathcal{X}$ is the closed embedding, which equals
$\mathcal{F}\otimes k(y)=\mathcal{F}/ \hat{\mathfrak{m}}_y\mathcal{F}$. This is an elementary fact in commutative algebra.
\end{rem}

The following will be used in the proof of Proposition \ref{global-B}.

\begin{prop}[]\label{Hodge-torsionfree} For a holomorphic family $\pi: \mathcal{X}\rightarrow Y$ of compact complex manifolds over a connected manifold of dimension one,
$h^1(X_y, \mathcal{O}_{X_y})$ is independent of $y\in Y$ if and only if the sheaf $R^2\pi_*\mathcal{O}_{\mathcal{X}}$ is torsion free.
\end{prop}
\begin{proof}
Remark that this type of results should be known to experts, such as \cite{s04} in a different context.  As it is crucial to our purpose here, we prefer to give a complete proof.

 The \lq if' part can be proved by the long exact sequence
\begin{equation}\label{tor-es}
  \cdots\rightarrow R^1\pi_*\mathcal{O}_{\mathcal{X}}\xrightarrow{\gamma} R^1\pi_*(\mathcal{O}_{\mathcal{X}}/{t\mathcal{O}_{\mathcal{X}}}) \xrightarrow{\delta} R^2\pi_*(t\mathcal{O}_{\mathcal{X}})\xrightarrow{\imath} R^2\pi_*\mathcal{O}_{\mathcal{X}}\rightarrow\cdots
\end{equation}
induced by the short exact sequence
$$0\rightarrow t\mathcal{O}_{\mathcal{X}}\xrightarrow{\imath} \mathcal{O}_{\mathcal{X}}\xrightarrow{\gamma} \mathcal{O}_{\mathcal{X}}/{t\mathcal{O}_{\mathcal{X}}} \rightarrow 0.$$
First note that by using Lemma \ref{2-5}
$$R^1\pi_*(\mathcal{O}_{\mathcal{X}}/{t\mathcal{O}_{\mathcal{X}}})\cong H^1(X_0,\mathcal{O}_{X_0}).$$
Now suppose that the Hodge number $h^1(X_y, \mathcal{O}_{X_y})$ is not constant.  We are easily reduced to the situation where the Hodge number $h^1(X_y, \mathcal{O}_{X_y})$ is constant around a punctured small connected neighborhood of some point $P\in Y$ and
jumps at this point $P$ by Lemma \ref{gct} and Proposition \ref{slia}. Set this point $P$ as $0$, and write $U$ for this small neighborhood and $\mathcal{X}$ still for $\pi^{-1}(U)$.  By using the first part of Lemma \ref{gct} and
(b) of Corollary \ref{exact-cri}, the preceding jumping property of $h^{0,1}$ enables us to choose an element $e\in H^1(X_0, \mathcal{O}_{X_0})$ not belonging
to the image of the map $\gamma$ in the long exact sequence \eqref{tor-es}.  Then $\delta(e)$ is nonzero in $R^2\pi_*(t\mathcal{O}_{\mathcal{X}})$ by the exactness
of the long exact sequence \eqref{tor-es}. Since the Hodge number $h^1(X_y, \mathcal{O}_{X_y})$ is constant over $U^*:=U\setminus \{0\}$, $R^1\pi_*(\mathcal{O}_{\mathcal{X}})$ is thus locally free over $U^*$ giving that $\gamma$
in \eqref{tor-es} is surjective outside $t=0$ by using Corollary \ref{ccs-bc}.  This gives in turn $\delta(e)=0$ outside $t=0$ by the long exact sequence \eqref{tor-es} again.

Next we check that $\delta(e)$ will give rise to a nontrivial torsion element of $R^2\pi_*\mathcal{O}_\mathcal{X}$.
Obviously, the map $\imath$ in \eqref{tor-es} induces an isomorphism outside $t=0$:
$$\imath: R^2\pi_*(t\mathcal{O}_{\mathcal{X}})\cong R^2\pi_*\mathcal{O}_\mathcal{X}.$$
Observe next that the map $\imath$ divided by $t$, denoted by $\jmath=\imath/t$,
induces an isomorphism over $U$
$$\jmath: R^2\pi_*(t\mathcal{O}_{\mathcal{X}})\cong R^2\pi_*\mathcal{O}_\mathcal{X}.$$
Recall that $0\ne\delta(e)\in R^2\pi_*(t\mathcal{O}_{\mathcal{X}})$ as just indicated.  Thus $\jmath(\delta(e))\ne 0$.
However $\jmath(\delta(e))=0$ outside $t=0$ since $\delta(e)=0$ outside $t=0$.  We now see that $\jmath(\delta(e))$ is a torsion element, as desired, by Proposition \ref{s-torsion} after possibly shrinking $U$ so that $R^2\pi_*\mathcal{O}_\mathcal{X}$ is also locally free on $U^*$.

 The \lq only if' part is a special case of \cite[Proposition in $\S 10.5.5$]{Grr}:
If $h^i(X_y,V|_{X_y})$ with a holomorphic vector bundle $V$ on $\mathcal X$ is independent of $y\in Y$, then the sheaf $R^{i+1}f_*V$ is torsion free.  Alternatively, one can prove this in a way similar to the \lq if' part by
the long exact sequence for some positive integer $m$
\begin{equation}\label{tor-es-m}
  \cdots\rightarrow R^1\pi_*\mathcal{O}_{\mathcal{X}}\xrightarrow{\gamma_m} R^1\pi_*(\mathcal{O}_{\mathcal{X}}/{t^m\mathcal{O}_{\mathcal{X}}}) \xrightarrow{\delta} R^2\pi_*(t^m\mathcal{O}_{\mathcal{X}})\xrightarrow{\imath} R^2\pi_*\mathcal{O}_{\mathcal{X}}\rightarrow\cdots
\end{equation}
associated with the short exact sequence
$$0\rightarrow t^m\mathcal{O}_{\mathcal{X}}\xrightarrow{\imath} \mathcal{O}_{\mathcal{X}}\xrightarrow{\gamma_m} \mathcal{O}_{\mathcal{X}}/{t^m\mathcal{O}_{\mathcal{X}}} \rightarrow 0$$
and by Grauert's base change theorem.
To start with, suppose that $R^2\pi_*\mathcal{O}_{\mathcal{X}}$ is not torsion free,
say $T(R^2\pi_*\mathcal{O}_{\mathcal{X}})_p\ne 0$ at some $p$, and that $R^2\pi_*\mathcal{O}_{\mathcal{X}}$ is locally
free on $V^*=V\setminus \{p\}$ for some neighborhood $V$ of $p$.
After possibly shrinking $V$, pick an $s\in \Gamma(V, R^2\pi_*\mathcal{O}_{\mathcal{X}})$ with $0\ne s_p\in T(R^2\pi_*\mathcal{O}_{\mathcal{X}})_p$.  Set this $p$ as $0$.
From  Proposition \ref{s-torsion} one sees that there exists an integer $m >0$ such that
\begin{equation}\label{vanishing}
t^m\cdot s =0\in \Gamma(V, R^2\pi_*\mathcal{O}_{\mathcal{X}}).
\end{equation}

We first assume $m=1$ in \eqref{vanishing} and work over \eqref{tor-es-m} with $m=1$.
Clearly $t\cdot s$ defines a section of $R^2\pi_*(t\mathcal{O}_{\mathcal{X}})$.
We claim that $t\cdot s$ is a nonzero section of $R^2\pi_*(t\mathcal{O}_{\mathcal{X}})$.
Recall the isomorphism
$$\jmath: R^2\pi_*(t\mathcal{O}_{\mathcal{X}})\cong  R^2\pi_*\mathcal{O}_{\mathcal{X}}$$
induced from $t\mathcal{O}_{\mathcal{X}} \rightarrow  \mathcal{O}_{\mathcal{X}}$ and division by $t$.
One has $\jmath(t\cdot s) =s$ and since $0\ne s\in \Gamma(V, R^2\pi_*\mathcal{O}_{\mathcal{X}})$ by construction, giving
$\jmath(t\cdot s)\ne 0$ hence $t\cdot s\ne 0$ in $R^2\pi_*(t\mathcal{O}_{\mathcal{X}})$ as claimed.
Moreover, by $\imath:  R^2\pi_*(t\mathcal{O}_{\mathcal{X}}) \rightarrow  R^2\pi_*\mathcal{O}_{\mathcal{X}}$ in the long exact sequence \eqref{tor-es}, one has
$\imath(t\cdot s)=t\cdot s\in \Gamma(V, R^2\pi_*\mathcal{O}_{\mathcal{X}})$ which is zero by \eqref{vanishing}
for $m=1$.
This means that $\ker\imath$ is nontrivial.  We are ready to show that the Hodge number $h^1(X_y, \mathcal{O}_{X_y})$
cannot be independent of $y\in V$, giving the conclusion of the  \lq only if' part for $m=1$.  Suppose otherwise.  Then the base change theorem in Corollary \ref{ccs-bc} implies that in the long exact sequence \eqref{tor-es-m},
the map $\gamma_{m=1}$ is a surjection hence that $\ker\imath$ is trivial. This contradicts the
preceding assertion that $\ker\imath$ is nontrivial.
We have now shown that $h^1(X_y, \mathcal{O}_{X_y})$ cannot be constant in $y$ as desired.

We can deal with the general case $m>1$ similarly.   The key is to replace $t$ by $t^m$ throughout the preceding paragraph and use the base change theorem for $m>1$ in Lemma \ref{gbc}. Indeed, if the Hodge number $h^1(X_y, \mathcal{O}_{X_y})$ is constant, the isomorphism in Corollary \ref{ccs-bc}  for $m=q=1$ implies the surjection in \eqref{gbciv} of Lemma \ref{gbc} and thus the isomorphism in \eqref{gbci} of Lemma \ref{gbc} for $m>1$, $q=1$ and  $M=\mathcal{O}_y/\mathfrak{m}_y^m$.  This yields  the desired surjection of the map $\gamma_m$
in \eqref{tor-es-m} by Lemma \ref{caresult} with $$\mathfrak{M}= {(R^1\pi_*\mathcal{O}_{\mathcal{X}})}_y,  \mathcal{A}=\mathcal{O}_y,  \mathfrak{a}=\mathfrak{m}_y^m.$$
As remarked earlier, this can lead to the desired conclusion similarly as done for $m=1$.
\end{proof}

\begin{prop}\label{global-B}
Let $\pi: \mathcal{X}\rightarrow \Delta$ be a holomorphic family of compact complex manifolds.
Let $B$ be an uncountable subset of $\Delta$ and for each $t$ of $B$, the fiber $X_t$ is equipped with a
line bundle $L_t$.  Assume that the deformation invariance of Hodge numbers
$h^{0,1}(X_t)$ is valid on $t\in \Delta$.
Then there exists a (global) line bundle $L$ over $\mathcal{X}$ such that
$L|_{X_{\tau}} = L_{\tau}$ for some $\tau\in B$ and that $c_1(L|_{X_s}) = c_1(L_s)$
for $s$ in some uncountable subset of $B$.
\end{prop}

\begin{proof}
The union of $c_1(L_t)$ in $H^2(X, \mathbb{Z})$ is countable,
but $B$ is uncountable. So there exists some uncountable subset $S$ of $B$ such that
if $s\in S$, then $c_1(L_s)$ is the same one in $H^2(X, \mathbb{Z})$, which we denote by the common $c$.

Let $h$ be the rank of $R^2\pi_*\mathcal{O}_{\mathcal{X}}$ at the generic point of $\Delta$.
By Proposition \ref{slia}, $R^2\pi_*\mathcal{O}_{\mathcal{X}}$ can be identified with
a vector bundle of rank $h$ on some $U$ of $\Delta$
with $\Delta\setminus U$ being a proper analytic subset of $\Delta$ which is a discrete subset of $\Delta$.
But we prefer to assume further that $\mathcal{O}_{\mathcal{X}}$ is cohomologically flat
on $U$ (in dimension $2$) which in our case means that $h^2(X_t, \mathcal{O}_{X_t})$ is independent of $t\in U$ after possibly shrinking $U$;
see Lemma \ref{gct} and Theorem \ref{Upper semi-continuity}.
One sees that the intersection $\hat{S}:= S \cap U$ remains uncountable.

For any $s$ of $\hat{S} \subseteq S$,
consider the commutative diagram
\begin{equation}\label{les}
\xymatrix@C=0.5cm{
  \cdots \ar[r]^{}
  & H^1(\mathcal{X}_U, \mathcal{O}^*_\mathcal{X})\ar[d]_{} \ar[r]^{}
  & H^2(\mathcal{X}_U, \mathbb{Z}) \ar[d]_{} \ar[r]^{i}
  & H^2(\mathcal{X}_U, \mathcal{O}_\mathcal{X})\ar[d]^{\Upsilon}\ar[r]^{} & \cdots \\
   \cdots \ar[r]
  &H^{1}({X_{s}},\mathcal{O}_{X_{s}}^*) \ar[r]^{}
  & H^{2}(X_{s},\mathbb{Z}) \ar[r]
  & H^{2}(X_{s},\mathcal{O}_{{X_{s}}})\ar[r]&\cdots.}
\end{equation}
We claim $i(c)=0$ in $H^2(\mathcal{X}_U, \mathcal{O}_\mathcal{X})$ where $i$ is induced by
$\mathbb{Z}\to \mathcal{O}_{\mathcal{X}}$.   Recall that $H^2(\mathcal{X}_U, \mathcal{O}_\mathcal{X})\cong \Gamma(U, R^2\pi_*\mathcal{O}_{\mathcal{X}})$ as obtained by the Leray spectral sequence argument in Theorem \ref{leray} using that $U$ is Stein.  We shall adopt them interchangeably in what follows.
For this claim, first note that the image of $i(c)$ under the map $\Upsilon$ to
$H^2(X_s, \mathcal{O}_{X_s})$, is zero
since $c$ is the first Chern class of the line bundle $L_s$ on $X_s$, $s\in \hat{S}$.
Recall above that $R^2\pi_*\mathcal{O}_{\mathcal X}$ is a vector bundle over $U$.  By Corollary \ref{ccs-bc},
Lemma \ref{2-5} and the cohomological flatness in our assumption, the fiber of this vector bundle at the closed point $s$ can be identified with $H^2(X_s, \mathcal{O}_{X_s})$, and the section $i(c)$ is a holomorphic section of this vector bundle.   Now that the holomorphic section $i(c)$ has zeros on $\hat{S}$ as just noted, and that $\hat{S}$
is uncountable, one concludes with Lemma \ref{section-zero} below that
$i(c)\equiv 0\in H^2(\mathcal{X}_U, \mathcal{O}_\mathcal{X})$, proving our claim above.

Fix any $p$ of $\Delta\setminus U$.  We shall prove that the preceding assertion $i(c)=0$ on $U$ yields $i(c)=0$ on a
neighborhood $V_p$ of $p$.  But this claim follows immediately if one applies Propositions  \ref{Hodge-torsionfree} and \ref{s-torsion} to a neighborhood $V_p$ of $p$ by using the deformation invariance of Hodge numbers $h^{0,1}(X_t)$ for all $t\in \Delta$ (noting that $\Delta\setminus U$ is a discrete subset of $\Delta$).

Combining these, we can conclude $i(c)=0$ in $\Gamma(\Delta, R^2\pi_*\mathcal{O}_{\mathcal{X}})$.
By $\Gamma(\Delta, R^2\pi_*\mathcal{O}_{\mathcal{X}})\cong H^2(\mathcal{X}, \mathcal{O}_\mathcal{X})$ again,
one sees that by the similar long exact sequence in the diagram \eqref{les} with $\mathcal{X}_U$ replaced by $\mathcal{X}$, $c$ can be the first Chern class of a global line bundle $L$ on $\mathcal{X}$.  That is,
$L|_{X_s}$ and $L_s$ have the same first Chern class $c$ for any $s$ of $\hat{S}$.   Fix any $\tau$ of $S$.  By  the same
argument as in Remark \ref{get-L0} due to Wehler, we can modify the global line bundle $L$ so that
$L|_{X_{\tau}}=L_{\tau}$ as line bundles on $X_{\tau}$.   This proves the proposition.
\end{proof}

The following well-known fact in complex analysis of one variable is used above:

\begin{lemma}[]\label{section-zero}
  Let $E$ be a holomorphic vector bundle over a one dimensional
(connected) complex manifold $M$ and $s$ a holomorphic section of $E$ on $M$.  If $s$ has uncountably
many zeros, then $s$ is identically zero.
\end{lemma}

Now, we shall improve the above result by the following: 
\begin{prop}\label{propa}
Let $M$ be a compact complex manifold with two holomorphic
line bundles $L_1$ and $L_2$ with $c_1(L_1)=c_1(L_2).$   Suppose that
$L_1$ is big.  Then $L_2$ is big too.
\end{prop}
\begin{proof}
First we note the following
characterization of a big line bundle $L_1$ over a compact complex manifold $M$: $L_1$ is big if and only if $L_1$ admits a singular Hermitian metric $h$ such that the curvature
current $\Theta_{L_1,h}$  is strictly positive, namely, $\Theta_{L_1,h} \geq \omega$ in the sense of
current for a smooth strictly positive $(1,1)$-form $\omega$ on $M$.  The proof of this characterization can be found
in \cite[Proposition 4.2]{Dem92}
for $M$ being K\"ahler and in \cite[Theorem 4.6]{JS} for general $M$.

For our purpose, let us first assume that $M$ is a $\partial\bar\partial$-manifold.  As
$L_2=L_1\otimes G$ for $G:=L_2\otimes L_1^*$ with $c_1(G)=0$, the $\partial\bar\partial$-lemma of $M$ gives rise to a
smooth hermitian metric $g$ on $G$ with its Chern curvature $\Theta_{G, g}=0$ (cf. \cite[\S\ 9.1]{zheng}).  From this $g$
together with the above characterization for bigness, Proposition \ref{propa} follows easily.

For the general case, by Lemma \ref{big-moi} our $M$ is Moishezon since $M$ admits a big line bundle $L_1$.  Moreover,
any Moishezon manifold satisfies the $\partial\bar\partial$-lemma by  \cite{par} or \cite[Theorem 5.22]{DGMS}.  Hence
our proposition follows from this and the preceding paragraph for the $\partial\bar\partial$-case.   Alternatively one can see this as follows. There exists a (finite sequence of) blow-up(s) $f: N\rightarrow M$ along smooth center(s)
with $N$ being a projective manifold.    It is true that $L_1$ is big on $M$ if and only if $f^*L_1$ is big on $N$, which follows from
\cite[Theorem 5.13]{Ue} (for the condition needed there we may assume that $kL_1$ is Cartier for
a large and fixed $k$) or from \cite[Corollary 6.13]{ryy}.   Now $N$ is obviously a $\partial\bar\partial$-manifold.
As $c_1(f^*L_1)=c_1(f^*L_2)$, it follows that $f^*L_2$ is big on $N$.  Thus $L_2$ is big on $M$.
The proof of Proposition \ref{propa} is completed.
\end{proof}

We can now apply Propositions \ref{global-B} and \ref{propa} to Conjecture \ref{conj-M}: 
\begin{thm}\label{thm-moishezon-update}
Let $B$ be an uncountable subset  of $\Delta$.
If the fiber $X_t:=\pi^{-1}(t)$ for each $t\in B$ is Moishezon
 and the deformation invariance of Hodge numbers $h^{0,1}(X_t)$ holds for all $t\in \Delta$, then $X_t$ is Moishezon for each $t\in \Delta$.  Moreover, there exists a bimeromorphic embedding $\Phi:\mathcal{X}\to \mathbb{P}^N\times\Delta$
for some $N\in \mathbb{N}$ such that
$\Phi(X_t)\subset\mathbb{P}^N\times \{t\}$. In fact, the exact statements as
\eqref{thm-gauduchon-update-ii} of Theorem \ref{thm-gauduchon-update} hold here too.
\end{thm}
\begin{proof} Suppose that the Moishezon fiber $X_t:=\pi^{-1}(t)$ for each $t\in B$ admits a big line bundle $L_t$.
Actually we have obtained a line bundle $L$ on $\mathcal{X}$ such that $c_1(L|_{X_t})=c=c_1(L_t)$ for each $t\in \hat{S}\subseteq B$ in (the last paragraph in the proof of) Proposition \ref{global-B} according to notations therein.
With Proposition \ref{propa}, we can avoid the use of Remark \ref{get-L0} (due to Wehler) and reach better
results.  That is, for each $t\in \hat{S}\subseteq B$, $L|_{X_t}$ is big since $L_t$ is assumed to be
big on $X_t$.   Now we are entitled to use Corollary \ref{unc-big} since $\hat{S}$ is uncountable and $L$ is
a global line bundle on $\mathcal{X}$.  It follows the conclusion that all the fibers $X_t$ are
Moishezon, as to be proved.  We postpone the part on the bimeromorphic embedding until the proof of Theorem \ref{thm-gauduchon-update}.\eqref{thm-gauduchon-update-ii} which actually also works for
the Hodge number case here.
\end{proof}

\subsection{Strongly Gauduchon metric case and bimeromorphic embedding}

Recall that Popovici introduced the following notion:
\begin{defn}[{\cite[Definition 4.1]{P3}}]\label{sGau}
Let $X$ be a compact complex $n$-dimensional manifold. A smooth positive-definite $(1, 1)$-form $\mathfrak{g}$ on $X$ will be said to be a \emph{strongly
Gauduchon metric} if the $(n, n-1)$-form $\partial\mathfrak{g}^{n-1}$ is $\bar\partial$-exact on $X$. If $X$ carries such a metric, $X$ will be said to be a \emph{strongly Gauduchon
manifold}.
\end{defn}

A nice deformation property of strongly Gauduchon manifold is the local stability under small deformation.
\begin{lemma}[{\cite[Theorem 3.1]{P14}}]\label{sdsGau}
Any small deformation of a strongly Gauduchon manifold is still a strongly Gauduchon manifold.
\end{lemma}
\begin{proof}
This result is implicit in the proof of \cite[Proposition 4.5]{P3}.
As shown in \cite[Proposition 4.2]{P3}, the existence of a strongly Gauduchon metric on a complex $n$-dimensional manifold $M$ is equivalent to the condition that there exists a real smooth $d$-closed $(2n-2)$-form $\Omega$ on $M$ with its $(n-1,n-1)$-type component $\Omega^{n-1,n-1}>0$  with respect to the complex structure on $M$. Now we come to the family $(X, J_t)_{t\in \Delta}$ of complex manifolds with a fixed differentiable manifold $X$.   We see that the $(n-1,n-1)$-type components of $\Omega$ with respect to
the complex structure $J_t$ vary smoothly in $t\in \Delta$ as the complex structures $(J_t)_{t\in \Delta}$ do. So one still has $\Omega_t^{n-1,n-1}>0$ on $(X, J_t)$ after possibly shrinking $\Delta$ about $0$. Then Michelsohn's procedure on extracting the root of order $n-1$ for $\Omega_t^{n-1,n-1}>0$ gives a smooth family $\{\mathfrak{g}_t\}_{t\in \Delta}$ of strongly Gauduchon metrics by \cite[Proposition 4.2]{P3} again. Thus,
$\partial_t\mathfrak{g}_t^{n-1}=-\bar\partial_t\Omega_t^{n,n-2}$ for any small $t\in \Delta$.
\end{proof}

We are now in a position to turn to the main part of this subsection.

\begin{lemma}\label{global-lemma2}
Let $\pi: \mathcal{X}\rightarrow \Delta$ be a holomorphic family of compact complex manifolds
and $B$ an uncountable subset of $\Delta$.  Suppose that  for each $t$ of $B$,
$X_t$ is equipped with a big line bundle $L_t$.
Then  there exists an open subset $U$ of $\Delta$ with $\Delta\setminus U$ being a proper complex analytic subset
of $\Delta$ and a line bundle $L$ over $\mathcal{X}_U=\pi^{-1}(U)$ such that
\begin{enumerate}[$(i)$]
    \item \label{global-lemma2-i}
$L|_{X_s}$ is big for any $s$ in an uncountable subset of $U$;
\item \label{global-lemma2-ii}
the similar estimate \eqref{16-3} holds on $U$.  As a consequence,
$X_t$ is Moishezon for every $t\in U$.
 \end{enumerate}
\end{lemma}
\begin{proof}  As in the proof of Proposition \ref{global-B},  $R^2\pi_*\mathcal{O}_{\mathcal{X}}$ is a vector bundle
over an open subset $U$ and there exists an uncountable subset $\hat{S}\subset U$ of $B$ such
that $c_1(L_t)=c$ are the same for all $t\in \hat{S}$.  Thus,
with the notations there, $i(c) =0$ on $U$ and it follows from the long exact sequence over $U$
that $c$ is the image of a global line
bundle $L$ over $\mathcal{X}_U$.

To prove \eqref{global-lemma2-i} that this line bundle $L$ satisfies the bigness on $X_s$ for $s\in \hat{S}$,
we just apply Proposition \ref{propa} to each $X_s$ for $s$ in $\hat{S}$ with the two line bundles $L|_{X_s}$ and $L_s$
because both have the same $c_1$ ($=c$).

The term \eqref{global-lemma2-ii} is just given by  \eqref{global-lemma2-i} together with
the proof of Proposition \ref{ki-dim-limit} applied to $U$.
\end{proof}

For a positive current $T$, we write $T_{ac}$ for the absolute continuous part
in its Lebesgue decomposition
(cf. \cite[Subsection 2.3]{Bou} for a brief review).   Recall also the \emph{volume} $v(L)$ of a big
line bundle $L$ over a compact complex manifold $M$ of dimension $n$:
$$ v(L)= \limsup_{k\to\infty}\frac{n!}{k^n}h^0(M, L^{\otimes k}).$$

\begin{prop}\label{abc}
Let $Y$ be a compact \emph{Moishezon} manifold of complex dimension $n$ with a big line bundle $J$.
Suppose that $Y$ is equipped with a smooth volume form $dV$ with $\int_YdV=1$.
Then for every constant $\varepsilon>0$, there exists a closed, positive $(1,1)$-current $T_J:= T_{J,\varepsilon}\in c_1(J)$ such that
$T_{J,ac}^n(x)> (v(J)-\varepsilon)dV(x)$ a.e. on $Y$.
\end{prop}
\begin{proof}
Compared with the proof of \cite[Theorem 1.2]{Bou}, our proof uses the degenerate Monge--Amp\`ere equations, but no weak
limit for  $T_{J,\varepsilon}$ with respect to $\varepsilon$ is taken since $Y$ is only assumed to be \emph{Moishezon} here.

As $Y$ is Moishezon,  there exists a modification
$\pi:\tilde Y\to Y$ where $\tilde Y$ is a projective manifold equipped with an ample divisor $A$ such that
$m\pi^*J=A+G$ for some effective divisor $G$ and some $m\in \mathbb{N}$ since $\pi^*J$ is still
big on $\tilde Y$ by \cite[Theorem 5.13]{Ue}.  Let the ramification divisor be denoted by
$R_{\pi}$ thus $K_{\tilde Y}=\pi^*K_Y+R_{\pi}$ and $h_R$ be a smooth Hermitian metric on $R_{\pi}$
as a line bundle.  Write $s$ for a canonical section of $R_{\pi}$ with $(s)=R_{\pi}$ and $||s||$ for its norm
with respect to $h_R$.
For the given volume form $dV>0$ on $Y$, $\pi^*dV$ denotes the pull-back on $\tilde Y$.
Let $\omega_{\tilde Y}$ be a K\"ahler form in $\frac{1}{m}c_1(A)$,
defined via a suitable Hermitian metric on $A$.    One sees that
$$d\tilde V:=\frac{\pi^*dV}{||s||^2}$$
is a smooth positive volume form on $\tilde Y$.  

To proceed further, we need to consider the {\it degenerate} Monge--Amp\`ere equation for $\tilde\omega$
\begin{equation}\label{c1A}
\tilde\omega^n=(\int_{\tilde Y}\omega_{\tilde Y}^n)\pi^*dV
\end{equation}
from the following nondegenerate  type as \cite[Section 5]{Yst}:
\begin{equation}\label{yau3}
 \tilde\omega_{\delta}^n =c_{\delta}(||s||+\delta)^2 (\int_{\tilde Y}\omega_{\tilde Y}^n) d\tilde V
\end{equation}
for $\tilde\omega_{\delta}:=\omega_{\tilde Y}+\sqrt{-1}\partial\bar\partial\varphi_{\delta}>0$ where
$\delta>0$ is small and the normalization constant $c_{\delta}$:
$$c_{\delta}=\int_{\tilde Y}\tilde\omega_{\delta}^n\cdot \Big(\int_{\tilde Y}\omega_{\tilde Y}^n\cdot\int_{\tilde Y}(||s||+\delta)^2 d\tilde V\Big)^{-1}$$
with $c_{\delta=0}=1$ since $\int_{\tilde Y}\pi^*dV=\int_Y dV=1$ by assumption.

Fortunately, \eqref{yau3} and its limiting version \eqref{c1A} have been solved by Yau in \cite[Theorems 1 and 3]{Yst}, so that
$ \tilde\omega_{\delta}>0$ exists and as $\delta\to 0$,  $\{\tilde\omega_{\delta}\}_{\delta}$
is uniformly bounded and converges in every compact subsets of
$\tilde Y\setminus R_{\pi}$ for a subsequence $\delta_k\to 0$ and  $\tilde\omega$ is obtained by taking this subsequence convergence. See also the monograph \cite{gz} for
a review and new results.

We claim that $\tilde\omega\in  \frac{1}{m}c_1(A)$.
Let $\chi$ be any smooth closed $(2n-2)$-form  on $\tilde Y$.
We note the following facts:
\item{i)} $\{|\partial_i\partial_{\bar j}\varphi_{\delta}|\}_{\delta}$ is uniformly bounded as shown \cite[pp. 374-375]{Yst},
\item{ii)}  $\partial_i\partial_{\bar j}\varphi_{\delta_k}\to\partial_i\partial_{\bar j}\varphi_{0}$ a.e. on
$\tilde Y$,
\item{iii)} $\int_{\tilde Y}\partial\bar\partial\varphi_{\delta}\wedge \chi=0$ since  $\partial\bar\partial\varphi_{\delta}$ is smooth and is a trivial class.

It follows that $\int_{\tilde Y}\partial\bar\partial\varphi_{0}\wedge d\chi'=0$ for any
smooth $(2n-3)$-form $\chi'$ by Lebesgue convergence theorem, hence that the $(1,1)$-form $\partial\bar\partial\varphi_{0}$ with locally integrable coefficients defines a closed current on $\tilde Y$.
To have $\tilde\omega\in  \frac{1}{m}c_1(A)$, it remains to show that
$\sqrt{-1}\partial\bar\partial\varphi_{0}$ gives a trivial class since $\tilde\omega=\omega_{\tilde Y}+\sqrt{-1}\partial\bar\partial\varphi_{0}$.    It suffices to show that
$\int_{\tilde Y}\partial\bar\partial\varphi_{0}\wedge \chi=0$.  This follows again by the same
convergence reasoning.   Our claim $\tilde\omega\in  \frac{1}{m}c_1(A)$ is proved.

To find a representative for $c_1(J)$, we first equip $G$ with a natural singular Hermitian metric $h_G$
such that $c_1(G, h_G)=[G]$, the closed  positive $(1,1)$-current defined by $G$ as \cite[Example 3.13]{dem}.
We have now a representative $\tilde\omega+\frac{1}{m}c_1(G, h_G):=T_{\pi^*J}$ for $c_1(\pi^*J)$ which
is also closed and positive.   Since $T_{\pi^*J}$ is now defined via a singular Hermitian metric, a well-known
procedure using the push-forward current $\pi_*T_{\pi^*J}$, which can also be defined
using a singular Hermitian metric on $J$, yields that  $\pi_*T_{\pi^*J}$ is closed, positive
and represents $c_1(J)$ according to \cite[proofs in Theorem 4.6]{JS}, which we denote by $T_J$.   Remark that we can't
assert that $T_{\pi^*J}$ is strictly positive, as $\tilde\omega$ is degenerate along $R_{\pi}$.

Write $T_{J, ac}$ for the absolute continuous part of $T_J$ as mentioned earlier.
We are going to compute $T^n_{J, ac}(x)$ for $x\in Y$.
Now $T_{J, ac}=\pi_*(T_{\pi^*J, ac})$ since $\pi$ is proper by \cite[Proposition 2.2]{Bou}.
Further one sees that $T_{\pi^*J, ac}=\tilde\omega$  as \cite[Subsection 2.3]{Bou} and \cite[Example 3.13]{dem}.
Since $\pi$ is a modification which is an isomorphism outside sets of measure $0$, it follows that
$(\pi_*(T_{\pi^*J, ac}))^n=\pi_*(T_{\pi^*J, ac}^n)$ according to \cite[p. 1054]{Bou} hence that
$T_{J, ac}^n=\pi_*(\tilde\omega^n)$ which equals $(\int_{\tilde Y}\omega_{\tilde Y}^n)\pi_*(\pi^*dV)$
by \eqref{c1A}.    By the fact that $\pi_*(\pi^*dV)=dV$ since $dV$ is smooth by the last equality on \cite[p. 40]{JS},
we have now achieved
$$T_{J, ac}^n=(\int_{\tilde Y}\omega_{\tilde Y}^n)dV.$$
It remains to make $\int_{\tilde Y}\omega_{\tilde Y}^n$ large to meet the theorem by choosing a suitable modification
$\tilde Y$.    Fortunately we are able to find such a suitable modification $\tilde Y\to Y$ provided the following result on
the Zariski decomposition in the sense of Fujita \cite{fu}.
\end{proof}

\begin{lemma}\label{fujita}
Let $L$ be a big line bundle on a compact complex manifold $Y$ of dimension $n$.
Then for every constant $\varepsilon>0$, there exists a modification
$\tilde Y\to Y$ and a decomposition $\pi^*L=A+E$ where $A$ is an ample $\mathbb{Q}$-divisor and $E$ an effective
$\mathbb{Q}$-divisor, such that $A^n>v(L)-\varepsilon$.
\end{lemma}

\begin{proof} When $Y$ is projective, this is the Fujita decomposition theorem \cite{fu} (cf. \cite[Theorem 14.6]{dem}).
In general we may first take a modification $\pi_1:\tilde Y_1\to Y$ where $\tilde Y_1$ is a projective manifold,
then $\pi_1^*L$ is big on
$\tilde Y_1$  by \cite[Theorem 5.13]{Ue}.    Now that $\tilde Y_1$ is projective, one has
a second modification $\pi_2:\tilde Y\to\tilde Y_1$ with decomposition $\pi_2^*(\pi_1^*L)=A+E$ as in the lemma such that
$A^n>v(\pi_1^*L)-\varepsilon$.   By setting $\pi=\pi_1\circ\pi_2$, the assertion follows if $v(\pi_1^*L)\ge v(L)$.
This holds true in view that $h^0(\tilde Y_1, q\pi_1^*L)\ge h^0(Y, qL)$
for every $q\ge 1$.  Remark that this inequality is actually an equality by \cite[Lemma 5.3]{Ue}.
\end{proof}

In the second part on bimeromorphic embedding of the proof for Theorem \ref{thm-gauduchon-update}, we will often apply the following preliminaries.

The first ones are proper modification and meromorphic map as shown in the nice reference \cite[$\S$ 2]{Ue} on bimeromorphic geometry, where all analytic spaces are assumed to be
 irreducible and reduced (but not necessarily compact) and a complex varieties are assumed to be compact, unless otherwise explicitly mentioned.

\begin{defn}[{\cite[Definition 2.1]{Ue}}]\label{modification}
A morphism $\pi: \tilde{X}\rightarrow X$ of two equidimensional irreducible and reduced complex spaces is called a \emph{proper modification}, if it satisfies:
\begin{enumerate}
  \item [$(i)$] $\pi$ is proper and surjective;
  \item [$(ii)$] there exist nowhere dense analytic subsets $\tilde E\subseteq \tilde{X}$ and $E \subseteq X$ such that
                  $$
                  \pi:\tilde{X}-\tilde E\rightarrow X-E
                  $$
                  is a biholomorphism, where $\tilde E:=\pi^{-1}(E)$ is called the \emph{exceptional space of the modification}.
\end{enumerate}
If $\tilde X$ and $X$ are compact, a proper modification $\pi: \tilde{X}\rightarrow X$ is often called simply a \emph{modification}.
\end{defn}

More generally, we have the following definition.
\begin{defn}[{\cite[Definition 2.2]{Ue}}]\label{bimero}
Let $X$ and $Y$ be two  irreducible and reduced  complex spaces.
A map $\varphi$ of $X$ into the power set of $Y$ is a \emph{meromorphic map} of $X$ into $Y$,
denoted by $\varphi: X\dashrightarrow Y$, if $X$ satisfies the following conditions:
\begin{enumerate}
  \item [$(i)$] The graph ${\mathcal{G}}({\varphi}):=\{(x,y)\in X\times Y: y\in \varphi(x)\}$ of $\varphi$ is an irreducible analytic subset in $X\times Y$;
  \item [$(ii)$] The projection map $p_X:{\mathcal{G}}({\varphi})\rightarrow X$ is a proper modification.
\end{enumerate}

A meromorphic map $\varphi: X\dashrightarrow Y$ of complex varieties is called a \emph{bimeromorphic map} if $p_Y:{\mathcal{G}}({\varphi})\rightarrow Y$ is also a proper modification.

If $\varphi$ is a bimeromorphic map, the analytic set
$$
\{(y,x)\in Y\times X: (x,y)\in {\mathcal{G}}(\varphi)\}\subseteq Y\times X
$$
defines a meromorphic map $\varphi^{-1}:Y\dashrightarrow X$ such that $\varphi\circ\varphi^{-1}=id_Y$ and $\varphi^{-1}\circ\varphi=id_X$.

Two complex varieties $X$ and $Y$ are called \emph{bimeromorphically equivalent} (or \emph{bimeromorphic}) if there exists a bimeromorphic map $\varphi: X\dashrightarrow Y$.

Obviously, the notion of bimeromorphic map is naturally valid for general complex spaces, cf. \cite[\S\ $2$ and $3$]{st} for example. More specially, see also the notions of rational and  birational maps  on
\cite[pp. 490-493]{Griffith} in the algebraic setting.
\end{defn}

Then we  need two more remarkable theorems.
\begin{thm}[{Theorem A of Cartan, \cite[Theorem 7.2.8]{h90}}]\label{cartan-a}
Let $\Omega$ be a Stein manifold and $\mathcal{F}$ a coherent analytic
sheaf on $\Omega$. For every $z\in \Omega$, the $\mathcal{O}_z$-module $\mathcal{F}_z$ is then generated by the
germs at $z$ of the sections in $\Gamma(\Omega, \mathcal{F})$.
\end{thm}

\begin{thm}[{Remmert's proper mapping theorem, \cite[Theorem 2.11 of Chapter III]{bs}}] \label{remmert}
The image of a closed analytic set by a proper morphism
is a closed analytic set.
\end{thm}

Finally, we will also be much concerned with questions of irreducibility and properness in the second part of the proof for Theorem \ref{thm-gauduchon-update} on bimeromorphic embedding and
refer the readers to \cite[\S\ $1$ and $2$ of Chapter $9$]{Grr} for a nice introduction.

The principal result of this section is the following, to which we devote ourselves in most
of the remaining paper.

\begin{thm}\label{thm-gauduchon-update}
Let $B$ be an uncountable subset  of $\Delta$.
Suppose that the fiber $X_t:=\pi^{-1}(t)$  is Moishezon for each $t\in B$
 and each fiber $X_t$ for $t\in \Delta\setminus B$ admits a strongly Gauduchon metric as in Definition \ref{sGau}.
Then we have the following.
\begin{enumerate}[$(i)$]
\item\label{thm-gauduchon-update-i}
Any $X_t$  for $t\in \Delta$ is a Moishezon manifold (of dimension $n$).
\item \label{thm-gauduchon-update-ii}
There exists a global line bundle $\tilde L$ on $\mathcal{X}$ such
that the restriction $\tilde L|_{X_t}$ is big for every $t\in \Delta$ ($\tilde L$ is usually termed \emph{$\pi$-big}).   Moreover, over $\Delta$, there exist a bimeromorphic map
$$\Phi:\mathcal{X}\dashrightarrow\mathcal{Y},$$
where
$\mathcal{Y}$ is a subvariety of $\mathbb{P}^N\times\Delta$ for some $N\in \mathbb{N}$ with every fiber $Y_t\subset\mathbb{P}^N\times\{t\}$ being a projective
 variety of dimension $n$, and also a proper analytic set $\Sigma\subset\Delta$  such that
$\Phi$ induces a bimeromorphic map
$$\Phi|_{X_t}:X_t\dashrightarrow Y_t$$
 for every
$t\in\Delta\setminus\Sigma$.
\end{enumerate}
\end{thm}
\begin{proof}

To start with, we will use Lemma \ref{global-lemma2} with the same notations therein.  By the estimate in \eqref{global-lemma2-ii} of
Lemma \ref{global-lemma2}, we assume without loss of generality that $h^0(X_t,L|_{X_t})$ is nonzero for any $t$ of $U$
and $L$ is $\pi$-big on $\mathcal{X}_U$.   Write $W=\Delta\setminus U$.  Fix
a point $p$ of $W$ with a small neighborhood $V_p$.  Since $W$ is a proper analytic subset of $\Delta$ which is
a discrete subset of $\Delta$, $V^*:=V_p\setminus\{p\}$ is contained in $U$ if $V_p$ is small.
Our first aim is to show that  $X_p$ is Moishezon for any such $p\not\in B$.

Here we work over $V_p$ and by restriction $L$ is a line bundle
on $\mathcal{X}_{V^*}$.   Recall that $\pi_*L$ is a coherent analytic sheaf on $V^*$.
By $h^0(X_t,L|_{X_t})\neq 0$ together with Theorem \ref{Upper semi-continuity}
and Lemma \ref{gct}, we see that $\pi_*L$ is nontrivial and on an open subset $V^*_1\subset V^*$,
$\pi_*L$ is cohomologically flat (in dimension $0$) in the sense of Lemma \ref{gct}, hence locally free.
There exists a global section $s$ of $\pi_*L$ on $V^*$ such that $s\not\equiv 0$ on $V^*_1$
by Theorem A of Cartan (= Theorem \ref{cartan-a}) since $V^*$ is Stein.

We are going to use the above $s$ to construct a divisor.  Since $s(t)$ is identified with a
section of $L|_{X_t}$ for each $t\in V_1^*$, we set $D_t$ to be
the divisor in $X_t$ defined by the section $s(t)$ if $s(t)\neq 0$.   Thus, we have a family of
divisors $\{D_t\}_{t\in V^*_2}$ where $V^*_2\subset V^*_1$ with the complement a discrete subset
of $V^*_1$.

Let $T_t$ be the
$d_t$-closed positive $(1, 1)$-current in $X_t$ defined by the divisors $D_t$.   Write $\omega$ for a smooth $2$-form on $\mathcal{X}$ in
$$H^2_{dR}(\mathcal{X},\mathbb{R})\cong H^2_{\mathcal{C}}(\mathcal{X},\mathbb{R})\cong H^2_{\mathcal{C}}(X_t,\mathbb{R}),$$
where $H^\bullet_{\mathcal{C}}(\bullet,\mathbb{R})$ means the de Rham cohomology group in the sense of current and the second isomorphism is induced by the embedding $i_t: X_t\rightarrow \mathcal{X}$.
Choose now $\omega$ in the de Rham class $c_1(L)$ so that as cohomology classes,
$$[\omega]|_{X_t}=[T_t]$$
and thus
$$\omega_t:=i^*_t\omega=T_t+d_t\beta_t$$
where $\beta_t$ is some real $1$-form on $X_t$ (in the sense of current).
Hence, by the proof, notations of Lemma \ref{sdsGau} and the strongly Gauduchon condition here, the above boundedness of
$$\int_{D_t}\mathfrak{g}_t^{n-1}=\int_{X_t}T_t\wedge\mathfrak{g}_t^{n-1}$$
is equivalent to that of
$$\int_{X_t}\partial_t \beta_t^{0,1}\wedge\mathfrak{g}_t^{n-1}
=-\int_{X_t} \beta_t^{0,1}\wedge\bar\partial_t\Omega_t^{n,n-2}
=-\int_{X_t} \bar\partial_t\beta_t^{0,1}\wedge\Omega_t^{n,n-2}
=-\int_{X_t} \omega_t^{0,2}\wedge\Omega_t^{n,n-2},$$
where $\omega_t^{0,2}$ denotes the $(0,2)$-part of $\omega_t$.   Here the integration of $\bar{\partial_t }\beta_t^{1,0}\wedge\mathfrak{g}_t^{n-1}$
is the complex conjugate of the above since $\mathfrak{g}_t$, $\omega_t$ and $\beta_t$ are real.   Finally, the boundedness of $\int_{X_t} \omega_t^{0,2}\wedge\Omega_t^{n,n-2}$ follows from that $\omega_t$ and $\Omega_t$ depend smoothly on $t$.
We conclude that the volume $vol(D_t)$ is uniformly bounded, hence that there exists a weakly convergent subsequence $T_{t_k}\to T_p$
for some $ t_k\to p$, $t_k\in V^*_2$ such that $T_p$ is a closed, positive current of type $(1,1)$ and
belongs to the same class as $c_1(L)$.

The difficulty lies in that the total divisor $D=\cup_tD_t$ is not known to be of bounded volume, which renders the
usual extension theorem of Bishop not applicable (cf. {\cite[Theorem (2.14)]{s74}} or \cite[Theorem 3]{bis}).
Instead, we may use the convergence theorem of Bishop \cite[Theorem 1]{bis} to
find a divisor $D_p\subset X_p$ after a limiting process of $D_t$, which does imply that the family $\{D_t\}_{t\in V_2^*}$
is relatively compact in the space of cycles in $\mathcal{X}$ in a suitable sense.  Unfortunately, these limiting divisors
are not yet known to be linearly equivalent to one another, despite all being of the same $c_1$ as that of $L$.  As such,
they do not immediately lead to an extension of $L$ across $p$.     We shall come to this again shortly.

Nevertheless,  if for each $p_k\in V^*$, we use the current
$T_{{\rm FY}, {p_k}}\in c_1(L|_{X_{p_k}})$ constructed
in Proposition \ref{abc} (the subscript FY may hint at Fujita--Yau), then by running through the same arguments as above, one observes that $\{T_{{\rm FY}, {p_k}}\}_{{p_k}\in V^*}$ are also bounded in mass, hence that they
give rise to a weakly convergent subsequence $T_{{\rm FY}, {p_k}}\to T_{{\rm FY}}$
for some ${p_k}\to p$.   Alternatively, we consider the Poincar\'e--Lelong equation (in the sense of current)
(cf. \cite[(3.11)]{dem}):
$$ \frac{\sqrt{-1}}{2\pi}\partial_t\bar\partial_t \log||s(t)||_{h_t}^2=T_t-T_{{\rm FY}, t}$$
where $h_t$ is a singular Hermitian metric on $L|_{X_t}$ for $T_{{\rm FY}, t}$ (see the proof of Proposition \ref{abc}, the
second-to-last paragraph).
Since $$\int_{X_t} \partial_t\bar\partial_t \log||s(t)||_{h_t}^2\wedge \mathfrak{g}_t^{n-1}
=\int_{X_t}\log||s(t)||_{h_t}^2\wedge \partial_t\bar\partial_t \mathfrak{g}_t^{n-1}=0$$
by the Gauduchon condition $\partial_t\bar\partial_t \mathfrak{g}_t^{n-1}=0$
and since the volume $vol(D_t)$ as $\int_{X_t}T_t\wedge\mathfrak{g}_t^{n-1}$
is uniformly bounded as just shown, we find that the mass $\int_{X_t}T_{{\rm FY}, t}\wedge\mathfrak{g}_t^{n-1}$
is uniformly bounded too.

In what follows, we have here normalized the volume $\int_{X_t}\mathfrak{g}_t^{n}=1$
for every $t\in V_p$.

We are ready to finish the proof that $X_p$ is Moishezon.
The following arguments correspond to those of Popovici in \cite[Step 2 on p. 527]{P3}, so let's just be brief.
Let's first note that the family $\mathcal{X}\to V_p$ can be equivalently described as
$(X, J_t)_{t\in V_p}$ where $X$ is a fixed differentiable manifold with $\{J_t\}_{t\in V_p}$
a family of complex structures holomorphically varied in $t$.   We will use this picture in this paragraph
without further notice.   As the same arguments on \cite[p. 527]{P3},
$T_{{\rm FY}}$ is a closed positive $(1,1)$-current and lies in an integral class.
We need the following semi-continuity property given by \cite[Proposition 2.1]{Bou}.  For almost every $x\in X_p$,
$$T_{{\rm FY}, ac}^n(x)\ge \limsup_{k\to\infty} (T_{{\rm FY}, {p_k}})^n_{ac}(x).$$
By the criterion of \cite[Theorem 1.3]{P08} or Theorem \ref{pop-hol}, it comes down to showing
that the RHS above is bounded from below by $\delta dV_p(x)$
for some constant $\delta>0$ where $dV_p=\mathfrak{g}_p^{n}$.
This requirement in our present situation can be fulfilled as an immediate consequence of
Proposition \ref{abc} and \eqref{global-lemma2-ii} of
Lemma \ref{global-lemma2}.  Namely, the uniform estimate in \eqref{global-lemma2-ii} of
Lemma \ref{global-lemma2} yields a $\delta>0$ such that the volumes
$$v(L|_{X_{p_k}})> 2\delta>0$$
 for all $k$,
so applying Proposition \ref{abc}, with a common $\varepsilon=\delta$, to every $L|_{X_{p_k}}$ does the job.

\begin{rem}\label{delta}
 The existence of the uniform lower bound $\delta$ is crucial here and also distinguishes our proof from \cite[Step 2 on p. 527]{P3} since $0\not\in\bar B$ is allowed and then one is unable to
take the subsequence limit ${p_k}\to p$ to obtain a uniform lower bound for  the volumes $v(L|_{X_{p_k}})$ as mentioned in the paragraph preceding Question \ref{question}.
\end{rem}
The first part of the theorem is proved.
We turn now to the second half of the theorem.

{\it Proof for \eqref{thm-gauduchon-update-ii} of Theorem \ref{thm-gauduchon-update}}.
Our proof is divided into four major steps.

\newtheorem{step}{Step}
\renewcommand{\thestep}{$($\Roman{step}$)$}
\begin{step}\label{step 1}
Global line bundle $L$ and Kodaira map $\Phi$
\end{step}
Since all fibers are now Moishezon by the first half of the theorem,
it is well-known that the deformation invariance of the Hodge number of all types holds for any such family,
so that in particular, $h^{0,1}(X_t)$ stays the same in $t\in \Delta$.   Our Proposition \ref{global-B} applies.
Jointly with Proposition \ref{propa} and Corollary \ref{unc-big}, we have arrived at
a global line bundle, denoted by $\tilde L$, on $\mathcal{X}$ such that $\tilde L|_{X_t}$ is big
for every $t\in \Delta$; this is said to be \emph{$\pi$-big} on $\mathcal{X}$.   Moreover,
the uniform estimate similar to \eqref{16-3} holds for $\tilde L|_{X_t}$ with every $t\in\Delta$.
The arguments in the beginning of this proof can now be refined as follows.

Let's  denote the above $\tilde L$ by $L$ if there is no danger of confusion.
To proceed further, a difficulty arises.  On each fiber $X_t$, there is some $q_t\in\mathbb{N}$ such
that $H^0(X_t, L|_{X_t}^{\otimes q_t})$ can induce a bimeromorphic embedding of $X_t$; $q_t$ may depend on $t$,
however.  We do not know how to control it even though the uniform estimate \eqref{16-3} holds here.
Our methods of overcoming the difficulty consist in the study of bimeromorphic geometry of
$\mathcal{X}\to\Delta$ as a family; the complex analytic desingularization \cite{ahv} also plays a useful role
in the process.

We can choose a point $b\in \Delta$ such that  for every $q\in \mathbb{N}$,
$h^0(X_t, L^{\otimes q}|_{X_t})$ is locally constant in some neighborhood (dependent on $q$) of $b$.
This is because for a given $m\in \mathbb{N}$,
the set of points in $\Delta$ where $h^0(X_t, L^{\otimes m}|_{X_t})$ fails to be locally constant,
is at most countable (cf. Theorem \ref{Upper semi-continuity}).
We fix such a $b\in \Delta$.  For this fiber $X_b$ at $b$, there exists a $\tilde{q}\in \mathbb{N}$
such that the Kodaira map associated with the complete linear system
$|L^{\otimes \tilde{q}}|_{X_b}|$ gives a bimeromorphic embedding of $X_b$ since $L|_{X_b}$ is big.
Now the preceding local invariance of $h^0(X_t, L^{\otimes \tilde{q}}|_{X_t})$ at $b$ yields
that the natural map
$$(\pi_*L^{\otimes \tilde{q}})_b\to H^0(X_b, L^{\otimes \tilde{q}}|_{X_b})$$
is surjective;
see Corollary \ref{ccs-bc}.  By Theorem A of Cartan (= Theorem \ref{cartan-a}),
one can choose $E$ linearly spanned by $$\{s_0, s_1, s_2,\cdots, s_N\}\subset \pi_*L^{\otimes \tilde{q}}(\Delta)$$ whose germs generate
$(\pi_*L^{\otimes \tilde{q}})_b$.   To sum up, by the identification $\pi_*L^{\otimes \tilde{q}}(\Delta)\cong
H^0(\mathcal{X}, L^{\otimes \tilde{q}})$,  the Kodaira map associated with the above $E$, denoted by
$$\Phi: \mathcal{X}\dashrightarrow \mathbb{P}^N\times\Delta: x\mapsto\big([s_0(x):s_1(x):\cdots :s_N(x)], \pi(x)\big),$$
is meromorphic
on $\mathcal{X}$ and bimeromorphic on $X_b$.  Here as usual, the concept of meromorphic/bimeromorphic maps is understood in the sense of Remmert; e.g. see Definition \ref{bimero} for more.   For this claim it is
stated in \cite[Example 2.4.2, p. 15]{Ue} for compact varieties.  If the target is a projective space,
it is treated in \cite[pp. 490-493]{Griffith} or \cite{s75};
a variant of which for our need is given as follows, since $\mathcal{X}$ is noncompact and  the target  is not a projective space,
 but $\mathbb{P}^N\times\Delta$.
Let $i$ denote the composite map $$\mathbb{P}^N\times\Delta\subset\mathbb{P}^N\times\mathbb{P}^1\to \mathbb{P}^{2N+1}$$
where the second map is the Segr\'e embedding $$([x_0:x_1:\cdots :x_N], [1:t])\mapsto [x_0:x_0t:x_1:x_1t:\cdots :x_N:x_Nt]$$
 (cf. \cite[p. 192]{Griffith}), and
 $$\tilde i:\mathcal{X}\times (\mathbb{P}^N\times\Delta)\to \mathcal{X}\times\mathbb{P}^{2N+1}:(z^1, z^2)\mapsto (z^1, i(z^2)).$$
Associated with the sections $$\{s_0, s_0\tilde t, s_1, s_1\tilde t, \cdots, s_N, s_N\tilde t\}\subset H^0(\mathcal{X}, L^{\otimes \tilde{q}}\otimes_{\mathcal{O}_\mathcal{X}}\pi^*\mathcal{O}_\Delta)$$
where $\tilde t\in H^0(\mathcal{X}, \pi^*\mathcal{O}_\Delta)$ denotes $\pi^*t$,
the composite map
$$\Phi_1:=i\circ \Phi:\mathcal{X}\dashrightarrow \mathbb{P}^{2N+1}$$
is a meromorphic map by \cite[pp. 490-492]{Griffith}
so that  with the graph  ${\mathcal{G}}_1\subset \mathcal{X}\times\mathbb{P}^{2N+1}$  of $\Phi_1$, one expects
$$\mathcal{G}:={\tilde i}^{-1}({\mathcal{G}}_1\cap \tilde i (\mathcal{X}\times(\mathbb{P}^N\times\Delta)))={\tilde i}^{-1}({\mathcal{G}}_1)\subset \mathcal{X}\times(\mathbb{P}^N\times\Delta)$$ to serve
as the graph of $\Phi$.  Indeed, this $\mathcal{G}$ is a complex space in $\mathcal{X}\times(\mathbb{P}^N\times\Delta)$.
The other conditions required for $\mathcal{G}$ as a graph variety of
a meromorphic map in Definition \ref{bimero} can
also be verified via ${\mathcal{G}}_1$.  It follows that
$$\Phi:\mathcal{X}\dashrightarrow \mathbb{P}^N\times\Delta$$
is a meromorphic map, as claimed.

Denote the respective projections by $$q_1:{\mathcal{G}}\to \mathcal{X}$$ and
$$q_2:=\pi_{\mathbb{P}^N\times\Delta}|_\mathcal{G}:{\mathcal{G}}\to \mathbb{P}^N\times\Delta$$
where $\pi_{\mathbb{P}^N\times\Delta}:\mathcal{X}\times (\mathbb{P}^N\times\Delta)\to \mathbb{P}^N\times\Delta$, and
set $$\mathcal{Y}:=q_2(\mathcal{G})=\Phi(\mathcal{X})\subset {\mathbb{P}^N\times\Delta}$$ (cf. \cite[p. 14]{Ue}).   The projection $\pi_{\mathbb{P}^N\times\Delta}$
is not proper.
But by Lemma \ref{proper} below, $\mathcal{Y}$ is a closed subvariety of $\mathbb{P}^N\times\Delta$
and as such $\mathcal{Y}\to\Delta$ is proper;
$$Y_t\subset\mathbb{P}^N\times\{t\}$$
 denotes the corresponding projective subvariety in $\mathcal{Y}$ seated at $t$.
 Since $\mathcal{X}$ is irreducible, so is $\mathcal{G}$ (cf. \cite[p. 13]{Ue}), and
$\mathcal{Y}=q_2(\mathcal{G})$ is thus irreducible. 
Clearly $\mathcal{Y}$ is of dimension $n+1$.

Remark that some refinements of the above construction will be made in Step \ref{step 3} for our need in due course.

Some tools in what follows have counterparts in algebraic category, however we work mostly within analytic category.
The above $\Phi$ is actually a morphism outside a subvariety $\mathcal{S}(\Phi)$ of codimension at least two in
$\mathcal{X}$ (cf. \cite[the third paragraph on p. 333]{re}), giving that for every $t\in \Delta$, $X_t\not\subset \mathcal{S}(\Phi)$ by dimension reason.
Outside the analytic set $S_t:=\mathcal{S}(\Phi)\cap X_t$ of codimension at least one in $X_t$, the restriction
$$\Phi_t:=\Phi|_{X_t\setminus S_t}$$ is a morphism.   By \cite[pp. 35-36]{st},
$\Phi_t$ is still a meromorphic map on $X_t$.  For later references, we can do it in the following way.
Recall that the projection $$q_1:\mathcal{G}(\Phi)={\mathcal{G}}\to\mathcal{X}$$ from the graph
of $\Phi$, is a proper modification 
(cf. Definition \ref{bimero}) so that $q_1^{-1}(X_t)\subset \mathcal{G}$ is an analytic subset of dimension
$n$.   Let  $\mathcal{C}$ be the unique irreducible component of $q_1^{-1}(X_t)$ which contains the graph of
$\Phi|_{X_t\setminus S_t}$, so that $\mathcal{C}\setminus q_1^{-1}(S_t)\cong X_t\setminus S_t$
biholomorphically and thus that $q_1|_{\mathcal{C}}:\mathcal{C}\to X_t$
is a proper modification (cf. Definition \ref{modification}).
Moreover, as for any subset $T\subset\mathcal{X}$,
$q_1^{-1}(T)=\mathcal{G}\cap(T\times \mathcal{Y})$, one has $\mathcal{C}\subset
q_1^{-1}(X_t)\subset X_t\times\mathcal{Y}$
hence that $\mathcal{C}$ induces a meromorphic map $\phi_{\mathcal{C}}:X_t\dashrightarrow\mathcal{Y}$
whose graph is precisely $\mathcal{C}$, since a meromorphic map is uniquely determined by an analytic subset $\mathcal{M}$ of $X_1\times X_2$ with two complex spaces $X_1,X_2$ which satisfies the condition $(M)$ that the projection $p_{X_1}:\mathcal{M}\rightarrow X_1$ is a proper modification as argued on \cite[p. 14]{Ue}.   Clearly $\phi_{\mathcal{C}}$ coincides
with $\Phi_t$ on the open subset $X_t\setminus S_t$.  Since $\phi_{\mathcal{C}}$ is meromorphic,
we now conclude that $\Phi_t:X_t\dashrightarrow \mathcal{Y}$ is a meromorphic map, as claimed.

We denote by $$\Phi|_{X_t}:X_t\dashrightarrow\mathcal{Y}$$ the meromorphic map associated with $\Phi_t$ as just shown.
 Since $\Phi|_{X_t}$ is now meromorphic
so that $\Phi|_{X_t}(X_t)$ is actually a closure of $\Phi|_{X_t}(X_t\setminus S_t)\subset Y_t$ in the
analytic set $Y_t$ (\cite[p. 493]{Griffith}),
we see that $\Phi|_{X_t}(X_t)\subset Y_t$ for every $t\in \Delta$.
However, it is not claimed that this closure $\Phi|_{X_t}(X_t)$ equals $Y_t$.
On the other hand, we will see in Step \ref{step 4} that with $\Phi$ in place of $\Phi|_{X_t}$,
$\Phi(X_t)= Y_t$ for every $t\in \Delta$.

We shall now see that $Y_t$ is of dimension $n$ for every $t\in \Delta$.
First note that, if $|b'-b|\ll 1$, $d\Phi|_{X_{b'}}$ is of rank $n$ at generic point of $X_{b'}$ since
it is so for $X_b$, hence that ${\rm dim}\, Y_b={\rm dim}\, Y_{b'}=n$.
In the algebraic setting, applying the upper semi-continuity of dimension (cf. \cite[Corollary 3 in Section 8 of Chapter 1]{mum}),
one is allowed to conclude that $Y_t$ is of dimension $n$.  Alternatively, an approach suitable in
our analytic setting is described as follows.
Let $$\{H_i\}_{1\le i\le n}\subset\mathbb{P}^N$$
be any hyperplane sections and
$$\tilde H_i:=H_i\times\Delta.$$
The closed subvariety $\tilde H_1\cap \tilde H_2\cap\cdots\cap\tilde H_n\cap \mathcal{Y}$ is projected,
via $\mathcal{Y}\subset \mathbb{P}^N\times\Delta\to\Delta$, down to a
subvariety $W_1\subset\Delta$ by Remmert's proper mapping theorem in complex spaces (= Theorem \ref{remmert}),
which is therefore the whole $\Delta$, as follows from the fact that
$W_1$ must contain a small open subset of $\Delta$ by ${\rm dim}\, Y_{b'}=n$.
If ${\rm dim}\,Y_{t_0}<n$ for some $t_0\in \Delta$, by choosing
$H_i$ in general positions such that $$H_1\cap H_2\cap\cdots\cap H_n\cap Y_{t_0}=\emptyset$$ with
$Y_{t_0}\subset \mathbb{P}^N$ via identification $ \mathbb{P}^N\cong  \mathbb{P}^N\times \{t_0\}$, then
with these $H_i$, $\tilde H_1\cap \tilde H_2\cap\cdots\cap\tilde H_n\cap \mathcal{Y}$ is projected,
via $\mathcal{Y}\to\Delta$, to a subset of $\Delta$
missing $t_0$, contradicting the preceding $W_1=\Delta$.   Hence ${\rm dim}\,Y_t\ge n$  for every $t\in\Delta$ so that
${\rm dim}\,Y_t=n$ since if
${\rm dim}\,Y_{t_1}\ge n+1$ for some $t_1\in\Delta$, it contradicts that
$\mathcal{Y}$ is irreducible and of dimension $n+1$ as
already indicated.

The following lemma has been used in the first half of this step. Notice that the projection $\mathcal{X}\times(\mathbb{P}^N\times\Delta)\to \mathbb{P}^N\times\Delta$ is usually not proper and we try to prove that its restriction to the graph of $\Phi$ is indeed proper.
\begin{lemma}\label{proper} With the notations as above, the projection morphism
$$q_2:\mathcal{G}\,(\subset \mathcal{X}\times(\mathbb{P}^N\times\Delta))\to \mathbb{P}^N\times\Delta$$ is
proper.  As a consequence, $q_2(\mathcal{G})$ is an analytic subvariety of $\mathbb{P}^N\times\Delta$.
\end{lemma}
\begin{proof} Since $\Phi:\mathcal{X}\dashrightarrow \mathbb{P}^N\times\Delta$ is a meromorphic map, $q_1:{\mathcal{G}}\to\mathcal{X}$
is a proper modification by Definition \ref{bimero} and  thus there exist two respective open dense
subsets $$U\subset \mathcal{G}\quad \text{and}\quad V\subset \mathcal{X},$$
which are biholomorphically equivalent under $q_1|_U: U\to V=q_1(U)$, such that $\Phi$ is a morphism on $V$
and $U=\{(x, \Phi(x)) \}_{x\in V}$ (cf. \cite[the remarks preceding Remark 2.3, p. 14]{Ue}).
Let $W_2\subset \mathbb{P}^N\times\Delta$ be a compact subset.  To prove by contradiction,
suppose that $q_2^{-1}(W_2)\subset \mathcal{G}$ is not compact.   Set the projections
$$\pi:\mathcal{X}\to\Delta\quad \text{and}\quad \pi_{\mathbb{P}^N\times\Delta}: \mathbb{P}^N\times\Delta\to \Delta.$$
Then under the projection $q_1:{\mathcal{G}}\to \mathcal{X}$,
$q_1(q_2^{-1}(W_2))$ is closed but not compact since $q_1$ is proper.
This means, since $\pi:\mathcal{X}\to\Delta$ is proper,
that there exists a sequence $t_k\in\Delta$ with $t_k\to \partial\Delta$ and
 $\{t_k\}_k\subset \pi(q_1(q_2^{-1}(W_2)))$.
We first show that $$\Phi(X_t)\subset \mathbb{P}^N_t:=\pi_{\mathbb{P}^N\times\Delta}^{-1}(t)$$ for every $t\in \Delta$.
See another more systematic approach in Step \ref{step 4} for this $\lq\lq\subset"$ again and the related issues.

As $\Phi(T)=\cup_{x\in T}\Phi(x)$ for a set $T\subset\mathcal{X}$,  we need to show that $\Phi(x)\in \mathbb{P}^N_t$
if $x\in X_t$. 
For any $(x, y)\in q_1^{-1}(x)\subset \mathcal{G}$,
there exists a sequence
$(x_j, y_j)\in U\subset \mathcal{G}$, $(x_j, y_j)\to (x, y)$ (by that $U$ is dense)
with $x_j\in V=q_1(U)$ and $y_j=\Phi(x_j)$.
Thus, $$\lim_j x_j=\lim_j q_1(x_j, y_j) \to q_1(x, y)=x.$$
By definition $\Phi(x)=q_2(q_1^{-1}(x))$ as in the first paragraph on \cite[p. 14]{Ue}, we have just seen that any point
$y\in \Phi(x)=q_2(q_1^{-1}(x))$
is a limit point of
the form $\Phi(x_j)=y_j$ for some sequence $x_j\to x$ in $\mathcal{X}$ with $x_j\in V$.
In this case, since  $x_j\not\in \mathcal{S}(\Phi)$, i.e., $\Phi$ is a morphism at $x_j$,
one has $\pi(x_j)=\pi_{\mathbb{P}^N\times\Delta}\Phi(x_j)$ by construction of $\Phi$, which is $\pi_{\mathbb{P}^N\times\Delta}(y_j)$.
Write $\pi(x_j)=\pi_{\mathbb{P}^N\times\Delta}(y_j)=t_j$.  So $y_j\in  \mathbb{P}^N_{t_j}$ and if $x\in X_t$,
then $t_j\to t$ since $\pi(x_j)\to \pi(x)$.
In short, for any $x\in X_t$ and any $y\in \Phi(x)$, $y=\lim_jy_j\in \lim_j\mathbb{P}^N_{t_j}$ which is
$\mathbb{P}^N_t$ as $t_j\to t$.
That is $\Phi(X_t)\subset \mathbb{P}^N_t$ for every $t\in \Delta$, as claimed.

The remaining is standard.
Corresponding to every
$t_k\in \{t_k\}_k$
above, there is a $$(x_k, y_k)\in q_2^{-1}(W_2)\subset \mathcal{G}$$ 
with $t_k=\pi(q_1(x_k, y_k))=\pi(x_k)$ so $x_k\in X_{t_k}$.
By $y_k\in \Phi(x_k)$, $$\pi_{\mathbb{P}^N\times\Delta}(y_k)\in \pi_{\mathbb{P}^N\times\Delta}(\Phi(x_k))\in \pi_{\mathbb{P}^N\times\Delta}(\mathbb{P}^N_{t_k})=t_k$$
by $x_k\in X_{t_k}$ and $\Phi(X_{t_k})\subset \mathbb{P}^N_{t_k}$ above.   In short, $\pi_{\mathbb{P}^N\times\Delta}(y_k)=t_k$.
By $y_k=q_2(x_k, y_k)$ and $q_2(x_k, y_k)\in W_2$, $\pi_{\mathbb{P}^N\times\Delta}(y_k)\in \pi_{\mathbb{P}^N\times\Delta}(W_2)$ which
 is a compact subset in $\Delta$ since $W_2\subset \mathbb{P}^N\times\Delta$ is compact by assumption.
This contradicts $\pi_{\mathbb{P}^N\times\Delta}(W_2)\ni\pi_{\mathbb{P}^N\times\Delta}(y_k)=t_k\to \partial\Delta$.
As said, the contradiction yields that $q_2:{\mathcal{G}}\to \mathbb{P}^N\times\Delta$ is
a proper morphism.

The second statement of the lemma follows from
Remmert's proper mapping theorem (cf. \cite[p. 395]{Griffith} or Theorem \ref{remmert}). 
\end{proof}

\begin{rem}\label{proper-rem} In fact, the above proof works for the following situation.
Let $\pi_{Z_1}:Z_1\to\Delta$ and $\pi_{Z_2}:Z_2\to \Delta$ be proper morphisms where
$Z_1$, $Z_2$ be irreducible (and reduced) complex spaces.  Suppose that $\psi: Z_1\dashrightarrow Z_2$ is a meromorphic map
(in the sense of Remmert)
with $\mathcal{G}(\psi)\subset Z_1\times Z_2$ the irreducible subvariety of the graph of $\psi$.  Let $\emptyset\neq U\subset Z_1$ be an open dense subset such that $\psi$ is a morphism on $U$.  Suppose furthermore that $\pi_{Z_1}|_U=\pi_{Z_2}\circ \psi|_U$.  Then
the projection morphism ${\mathcal{G}}(\psi)\to Z_2$ is proper.
\end{rem}

The remaining proof is devoted to the bimeromorphic problem of $\Phi$ and $\Phi|_{X_t}$.
\begin{step}\label{step 2}
Bimeromorphic embedding of $\Phi$
\end{step}  We shall use the notations
$$\Phi:\mathcal{X}\dashrightarrow\mathbb{P}^N\times\Delta\quad \text{and}\quad \Phi:\mathcal{X}\dashrightarrow\mathcal{Y}=\Phi(\mathcal{X})$$
interchangeably.   Let's start with a desingularization
$\mathcal{R}_{{\mathcal{Y}}}:\tilde{\mathcal{Y}}\to\mathcal{Y}$; it can be chosen as a proper modification.
For a review, see \cite[Theorem 2.12]{Ue} for compact cases and \cite[Theorem 7.13]{pe}
or \cite[Theorem 5.4.2, p. 271]{ahv} for general cases.
Note that if $h$ is a proper modification between complex spaces,
then $h^{-1}$ is still a meromorphic map \cite[p. 34]{st}
and hence $h$ is a bimeromorphism \cite[p. 33]{st}.
Write
$$\Psi:=\mathcal{R}_{{\mathcal{Y}}}^{-1}\circ\Phi:\mathcal{X}\dashrightarrow \mathcal{\tilde Y},$$
 which is still meromorphic (cf. \cite[pp. 16-17]{Ue}).
In algebraic cases,
if $g:X\to Y$ with $Y$ irreducible is a generically finite morphism such that
$g(X)$ is dense in $Y$ or equivalently $g$ is dominant, then there exists an open dense subset $U\subseteq Y$
such that the induced morphism $g^{-1}(U)\to U$ is a finite morphism (e.g. \cite[Exercise 3.7 of Chapter II]{Ht}).
If $g$ is only a generically finite dominant {\it rational} map, by going to its graph and restricting to the open dense subset $V\subset X$ which is the complement of the indeterminacies of $g$, one is reduced to the morphism case and
there is a similar conclusion.

Analytically, let's adopt a similar strategy here.
Write $\mathcal{S}(\Psi)$ for the indeterminacies of $\Psi$,
which is of codimension at least two in $\mathcal{X}$ (\cite[the third paragraph on p. 333]{re}).
Since $\mathcal{X}$ and $\tilde{\mathcal{Y}}$ are smooth and of the same dimension,
the ramification divisor $R_\Psi\subset {\mathcal{X}}$ is
well-defined. Namely, it is first defined outside $\mathcal{S}(\Psi)$ and then extends across it since ${\rm dim}\,R_\Psi>{\rm dim}\,\mathcal{S}(\Psi)$ by
Remmert--Stein extension theorem, cf. \cite[p. 293]{bis}.
Let
$$\mathcal{G}(\Psi)\subset \mathcal{X}\times\mathcal{\tilde Y}$$
denote
the graph of $\Psi$ with the projections $$p_{\mathcal{X}}:\mathcal{G}(\Psi)\to\mathcal{X}\quad \text{and}\quad
p_{\mathcal{\tilde Y}}: \mathcal{G}(\Psi)\to \mathcal{\tilde Y},$$ respectively.
Having proved Lemma \ref{py-proper} below that $p_{\mathcal{\tilde Y}}: \mathcal{G}(\Psi)\to \mathcal{\tilde Y}$ is proper, one knows that the image of an analytic set under $p_{\mathcal{\tilde Y}}$
is still analytic by the proper mapping theorem of Remmert (= Theorem \ref{remmert}).

Set $$\Psi(\mathcal{S}(\Psi)):
=p_{\tilde{\mathcal{Y}}}(p_{\mathcal{X}}^{-1}(\mathcal{S}(\Psi)))$$ which is a proper  analytic subvariety of $\tilde{\mathcal{Y}}$ and
similarly the subvariety
$$\Psi^{-1}(\Psi(\mathcal{S}(\Psi))):=p_{\mathcal{X}}(p_{\tilde{\mathcal{Y}}}^{-1}(\Psi(\mathcal{S}(\Psi))))\subset\mathcal{X};$$
also subvarieties $\Psi(R_\Psi)$, $\Psi^{-1}(\Psi(R_\Psi))$.
Write $$\Psi':\mathcal{X}\setminus \big(\Psi^{-1}(\Psi(\mathcal{S}(\Psi)))\cup \Psi^{-1}(\Psi(R_\Psi))\big)=:\mathcal{X}'
\to {\tilde{\mathcal{Y}}}':=\tilde{\mathcal{Y}}\setminus \big(\Psi(\mathcal{S}(\Psi))\cup \Psi(R_\Psi)\big)$$ and $\Psi'_t$
for its restriction to (an open part of) $X_t$, more precisely to
$X_{t}':=X_{t}\cap \mathcal{X}'$ with images in $\tilde Y'_{t}:={\tilde Y}_{t}\cap \tilde{\mathcal{Y}}'$
for those $X_{t}'\neq\emptyset$. Here $\tilde Y_t:=\pi_{\tilde{\mathcal{Y}}}^{-1}(t),$
where $$\pi_{\tilde{\mathcal{Y}}}:\tilde{\mathcal{Y}}\to \Delta$$ is the projection via $\tilde{\mathcal{Y}}\to\mathcal{Y}\to\Delta$.
By construction $\Psi'$ is a surjective morphism and since $d\Psi'$ is now of maximal rank everywhere,
$\Psi'$ is a local biholomorphism. 
The $\mathcal{X}'$ and ${\tilde{\mathcal{Y}}}'$ can possibly be enlarged.  Suppose that
$x\in \mathcal{X}\setminus \mathcal{S}(\Psi)$ and $\Psi$ is a local
biholomorphism between the open neighborhoods $\mathcal{H}_x\ni x$ and $\mathcal{K}_{\Psi(x)}\ni \Psi(x)$.  Then
$\Psi|_{\mathcal{X}'\cup \mathcal{H}_x}:\mathcal{X}'\cup \mathcal{H}_x\to {\tilde{\mathcal{Y}}}'\cup \mathcal{K}_{\Psi(x)}$ is still surjective and a local biholomorphism.
By enlarging $\mathcal{X}'$ and ${\tilde{\mathcal{Y}}}'$ this way, we can assume that if $\Psi$ is a local biholomorphism
at $x'$, then $x'\in \mathcal{X}'$ and $\Psi(x')\in{\tilde{\mathcal{Y}}}'$.  Here, $\mathcal{X}'\subset\mathcal{X}$ and
${\tilde{\mathcal{Y}}}'\subset{\tilde{\mathcal{Y}}}$ are connected open dense subsets in the sense of ordinary complex topology.

We would like to show that $\Psi'$ is a finite morphism.    First note that since $\Psi'$ is a morphism and surjective,
$$\Psi'(X_t')=\tilde Y_t'=\Psi_t'(X_t')\ \text{and}\ \Psi'^{-1}(\tilde Y_t')=X_t'=\Psi_t'^{-1}(\tilde Y_t').$$
It follows that, if $\Psi'^{-1}(\Psi'(x'))$ is infinite for some $x'\in\mathcal{X}'$ with
$y_\tau=\Psi'(x')\in  \tilde Y_\tau'\subset {\tilde{\mathcal{Y}}'}$, 
then ${\Psi'_\tau}^{-1}(y_\tau)=\Psi'^{-1}(y_\tau)\subset X_\tau$ is also infinite.
This cannot occur.
We shall now see that $$\Psi|_{X_\tau}:X_\tau\to  \Psi|_{X_\tau}(X_\tau)\subset \tilde Y_\tau$$ is generically finite,
and that, with $\Psi'_\tau$ defined on $X_\tau':=X_\tau\cap\mathcal{X}'$,
$$C:= {\Psi'_\tau}^{-1}({\Psi'_\tau}(x_{\tau}'))$$ is necessarily finite if $x_{\tau}'\in X_\tau'$.
Here $\Psi|_{X_\tau}$ is a meromorphic map by the same reasoning that $\Phi|_{X_t}$ is meromorphic for every $t\in \Delta$, in Step \ref{step 1}.
This will prove our claim that $\Psi'$ is a finite morphism.

To work on the meromorphic map $\Psi|_{X_\tau}:X_\tau\dashrightarrow \Psi|_{X_\tau}(X_\tau)$ above, we consider the meromorphic map
$$\widehat{\Psi|_{X_\tau}}:=p_2^{-1}\circ \Psi|_{X_\tau}\circ p_1:\hat X_\tau\dashrightarrow \hat Z_{\tau},$$
where
$$p_1:\hat X_\tau\to X_\tau\quad \text{and}\quad p_2:\hat Z_{\tau}\to \Psi|_{X_\tau}(X_\tau)$$
are proper modifications from projective
manifolds $\hat X_\tau$ and $\hat Z_{\tau}$, respectively. Here $\Psi|_{X_\tau}(X_\tau)$ is a Moishezon manifold by \cite[Corollary 2.24]{cp}. Now that $\widehat{\Psi|_{X_\tau}}$ is generically finite, as well-known
since $\widehat{\Psi|_{X_\tau}}$ is an algebraic (rational) map and is generically of maximal rank,
one has that $\Psi|_{X_\tau}=p_2\circ \widehat{\Psi|_{X_\tau}}\circ p_1^{-1}$ is also generically finite. 
To verify that $C(={\Psi'_\tau}^{-1}({\Psi'_\tau}(x_{\tau}')))$ is finite, suppose otherwise that $C$ is infinite.
We can assume, by further monoidal transformations of $\hat X_\tau$, that
$\widehat{\Psi|_{X_\tau}}:\hat X_\tau\to \hat Z_{\tau}$ is actually a morphism.
By commutativity ${\Psi'_\tau}\circ p_1=p_2\circ \widehat{\Psi|_{X_\tau}}$ where defined,  thus $$\widehat{\Psi|_{X_\tau}}^{-1}(p_2^{-1}(K))\supset
p_1^{-1}({\Psi'_\tau}^{-1}(K))$$ for any set $K\subset \Psi|_{X_\tau}(X_\tau)$, one sees
that $$p_1^{-1}(C)\subset \widehat{\Psi|_{X_\tau}}^{-1}(p_2^{-1}({\Psi'_\tau}(x_{\tau}')))=:\hat C$$
with $\hat C$ being an analytic set in $\hat X_\tau$.  Since $\hat C\subset \hat X_\tau$ has only finitely many
irreducible components, there must exist an irreducible component $\mathcal{J}\subset
\hat X_\tau$ of $\hat C$ such that
$p_1(\mathcal{J})\cap C$ is infinite.  The fact $|p_1(\mathcal{J})|=\infty$ implies that the irreducible analytic set $p_1(\mathcal{J})$ in $X_\tau$
must be of dimension at least one since $X_\tau$ is compact.
Recall $X_\tau'$ and ${\Psi'_\tau}$ above. 
This $p_1(\mathcal{J})\cap X_\tau'\supset p_1(\mathcal{J})\cap C$, a nontrivial open subset of $p_1(\mathcal{J})$,
is also of dimension at least one since $p_1(\mathcal{J})$ is irreducible. 
By the commutativity ${\Psi'_\tau}\circ p_1=p_2\circ\widehat{\Psi|_{X_\tau}}$,
$${\Psi'_\tau}(p_1(\mathcal{J})\cap X_\tau')\subset p_2(\widehat{\Psi|_{X_\tau}}(\mathcal{J}))
\subset p_2(\widehat{\Psi|_{X_\tau}}(\hat C)) \subset p_2(p_2^{-1}({\Psi'_\tau}(x_{\tau}')))={\Psi'_\tau}(x_{\tau}'),$$ i.e.,
$p_1(\mathcal{J})\cap X_\tau'\subset  {\Psi'_\tau}^{-1}({\Psi'_\tau}(x_{\tau}'))=C$ which by ${\rm dim}(p_1(\mathcal{J})\cap X_\tau')\ge 1$, gives that
$C\subset X_\tau'$, must also be of dimension at least one.  The facts that ${\Psi'_\tau}(C)$ is a point
and ${\rm dim}\, C\ge 1$ contradict that ${\Psi'_\tau}$ on $X_\tau'$ is a local biholomorphism.
As said, this contradiction proves that $\Psi':\mathcal{X}'\to \tilde{\mathcal{Y}}'$ is a finite morphism.

Given that $\Psi':\mathcal{X}'\to \tilde{\mathcal{Y}}'$ is finite, let's recall that in algebraic cases,
a finite surjective morphism between nonsingular varieties
over algebraically closed field is flat (cf. \cite[Exercise 9.3 of Chapter III]{Ht}) and a finite flat morphism $g:X\to Y$
with $Y$ Noetherian gives
that $g_*\mathcal{O}_X$ is a locally free $\mathcal{O}_Y$-module (cf. \cite[Proposition 2 in Section 10 of Chapter 3]{mum}). 
Applying these algebraic facts to $\Psi'$, 
one sees that 
$\Psi'_*\mathcal{O}_{\mathcal{X}'}$ is locally free, say, of rank $r$ on ${\tilde{\mathcal{Y}'}}$.
The local freeness of $\Psi'_*\mathcal{O}_{\mathcal{X}'}$ yields that the cardinality $(=r)$ of
$\Psi'^{-1}(y)$ is independent of $y\in \tilde{\mathcal{Y}}'$.

We shall now prove a similar result as above in the present analytic setting.  This is standard in covering spaces of topology.
For notations and references in the later use, we give some details.
Let $x_0'\in X_{t_0}'$ and $x'\in \mathcal{X}'$ be any point nearby $x_0'$.
Connecting the two points $\Psi'(x_0')=y_0'$ and $\Psi'(x')=y'$ by an analytic curve $\check{C}\subset\mathcal{Y}'$,
one has that $\tilde C:=\Psi'^{-1}(\check{C})$ is an
analytic set with each irreducible component $\tilde C_i$, $1\le i\le k$, being of dimension one and
$\Psi'(\tilde C_i)=\check{C}$ since $\Psi'$ is finite
surjection as remarked above.   Further, these components are pairwise disjoint.  For, if $c\in \tilde C_i\cap \tilde C_j$ with
$\tilde C_i\ne \tilde C_j$, then $\Psi'$ would be seen to be at least ``two-to-one" around $c$, violating the fact
that $\Psi'$ is a local biholomorphism everywhere, as mentioned earlier.   This property of disjointness leads us to arrive at
the fact that each point in
$\Psi'^{-1}(y_0')$ is joined by a unique $\tilde C_i$ to a point in $\Psi'^{-1}(y')$ and vice versa, so that
$\Psi'^{-1}(y_0')$ and $\Psi'^{-1}(y')$ are of the same cardinality $(=k)$.
If the two points $x_0', x' \in \mathcal{X}'$ are not close to each other,
the same result remains valid since $\mathcal{X}'$ is connected.
We conclude that, if with some $t_0$ it holds that $\Psi'_{t_0}$ is injective, then
$\Psi'$ is injective too since ${\Psi_{t_0}'}^{-1}(\tilde Y_{t_0}')=\Psi'^{-1}(\tilde Y_{t_0}')$ as
previously given.

We are going to show that $\Psi'$ is a biholomorphism.  The surjection part is
noted earlier; the injection part is to see that $X_b':=
X_b\cap \mathcal{X}'\ne\emptyset$. Recall that
$\Phi|_{X_b}$ is bimeromorphic by the assumption on $X_b$,
hence that when $X_b'\ne \emptyset$, $\Psi'_{b}$ is injective,
giving that $\Psi'$ is injective by the preceding paragraph.
Fix a $x_b\in X_b$ such that $\Phi|_{X_b}$ is a morphism
at $x_b$ and that $d(\Phi|_{X_b})(x_b)$ is of rank $n$.
We see that $d\Phi$ is always nonsingular
along the $t$-direction, hence that $d\Phi(x_b)$ is of rank $n+1$.  So
$\Phi:\mathcal{X}\to\mathcal{Y}$ is a local immersion at $x_b$; it is a local biholomorphism at $x_b$
provided that $\mathcal{Y}$ is smooth at $\Phi(x_b)$.   To facilitate our discussion, we make the
claim that
$$\hbox{{\it
$\mathcal{Y}$ is smooth at $\Phi(x_b)$, and thus $\Phi:\mathcal{X}\to \mathcal{Y}$ is  a local biholomorphism at $x_b$.}}  $$ 
We come for a complete discussion of this claim in Step \ref{step 3}.

It is a remarkable fact that the desingularization $\mathcal{R}_{{\mathcal{Y}}}:\tilde{\mathcal{Y}}\to \mathcal{Y}$
can be chosen in such a way that it is an isomorphism away from the singular points of $ \mathcal{Y}$ (\cite[Theorem 5.4.2,
p. 271]{ahv}). 
Let's choose such a desingularization in advance.  It follows that $\mathcal{R}_{{\mathcal{Y}}}^{-1}$ is a local
biholomorphism at $\Phi(x_b)$ since ${\mathcal{Y}}$ is assumed smooth there.  So
$\Psi=\mathcal{R}_{{\mathcal{Y}}}^{-1}\circ\Phi$ is a local biholomorphism at $x_b$ and thus $x_b\in \mathcal{X}'$
by construction of $\mathcal{X}'$, giving $x_b\in X_b'$ so $X_b'\ne\emptyset$.
As remarked earlier,  $X_b'\ne\emptyset$ implies that $\Psi'$ is injective
and, in turn, that $\Psi'$ is a biholomorphism.

Now since $\Psi'$ is proved to be biholomorphic, with the fact that
$\Psi$ is meromorphic, we see that $\Psi$ is a bimeromorphic map.   For, the set
$$\widehat{\mathcal{G}(\Psi)}:=\{(y, x)\in \tilde{\mathcal{Y}}\times\mathcal{X}: (x, y)\in \mathcal{G}(\Psi)\}$$ is irreducible and analytic since
$\mathcal{G}(\Psi)$ is so.  The projection morphism
$$\hat p_{\tilde{\mathcal{Y}}}:\widehat{\mathcal{G}(\Psi)}\to \tilde{\mathcal{Y}}$$
is proper as shown in Lemma \ref{py-proper}, and
since $\Psi'$ is biholomorphic and then
$\hat p_{\tilde{\mathcal{Y}}}$ induces $$\mathcal{G}(\Psi'^{-1})\cong \tilde{\mathcal{Y}}'\subset\tilde{\mathcal{Y}},$$
where $\mathcal{G}(\Psi'^{-1})\subset \widehat{\mathcal{G}(\Psi)}$ is open and dense, $\hat p_{\tilde{\mathcal{Y}}}$ is thus a proper modification.
It follows from Definition \ref{modification}, with the fact that
$\hat p_{\tilde{\mathcal{Y}}}$ is a proper modification, that $\Psi^{-1}:\tilde{\mathcal{Y}}\dashrightarrow\mathcal{X}$ is a meromorphic map by Definition \ref{bimero}
and hence a bimeromorphic map since $\Psi\circ \Psi^{-1}={\rm id}_{\tilde{\mathcal{Y}}}$ and $\Psi^{-1}\circ \Psi={\rm id}_{\mathcal{X}}$ (or see \cite[p. 33]{st}).
Having that $\Psi:\mathcal{X}\dashrightarrow\tilde{\mathcal{Y}}$ is a bimeromorphic map, we now know
that $\Phi=\mathcal{R}_{{\mathcal{Y}}}\circ \Psi$ is a bimeromorphic map since the composite of two bimeromorphic maps remains meromorphic
hence bimeromorphic (cf. \cite[pp. 16-17]{Ue} or \cite[2) of Proposition 9]{st}),
as claimed in the second part of the theorem.

\begin{lemma}\label{py-proper}
The map $p_{\mathcal{\tilde Y}}: \mathcal{G}(\Psi)\to \mathcal{\tilde Y}$ is proper.
\end{lemma}
\begin{proof}
  This is equivalent to that $\hat p_{\tilde{\mathcal{Y}}}:\widehat{\mathcal{G}(\Psi)}\to \tilde{\mathcal{Y}}$ is proper
in the preceding paragraph. To justify our assertion above that $\hat p_{\tilde{\mathcal{Y}}}$ is proper, noting that the spaces under consideration are not
compact and the target space $\tilde{\mathcal{Y}}$ may be more general than those in Step \ref{step 1},
one uses Remark \ref{proper-rem}.  Alternatively, let's indicate arguments while dropping most details. 
Let $W_3\subset \tilde{\mathcal{Y}}$ be a compact set and suppose that $\hat p_{\tilde{\mathcal{Y}}}^{-1}(W_3)\subset\widehat{\mathcal{G}(\Psi)}$ is not compact.
Then under the projection $\hat p_{\mathcal{X}}:\widehat{\mathcal{G}(\Psi)}\to \mathcal{X}$,
$\hat p_{\mathcal{X}}(\hat p_{\tilde{\mathcal{Y}}}^{-1}(W_3))$ is closed but not compact, so that 
there exists a sequence $t_k\in\Delta$ with $t_k\to \partial\Delta$ and
 $\{t_k\}_k\subset \pi(\hat p_{\mathcal{X}}(\hat p_{\tilde{\mathcal{Y}}}^{-1}(W_3)))$.  As in Lemma \ref{proper}, to
show that for every $t\in \Delta$ $$\Psi(X_t)\subset \tilde Y_t=\pi_{\tilde{\mathcal{Y}}}^{-1}(t),$$
where $\pi_{\tilde{\mathcal{Y}}}:\tilde{\mathcal{Y}}\to \Delta$ is the projection via $\tilde{\mathcal{Y}}\to\mathcal{Y}\to\Delta$, 
we are reduced to showing that $\Psi(x)\in \tilde{Y_t}=\pi_{\tilde{\mathcal{Y}}}^{-1}(t)$ if $x\in X_t$. 
Now that $\hat p_{\mathcal{X}}:\widehat{\mathcal{G}(\Psi)}\to\mathcal{X}$ being a modification, is a biholomorphism
between the dense and open subsets $U$, $V$ of $\mathcal{G}$, $\mathcal{X}$, respectively,
with $V\cap \mathcal{S}(\Psi)=\emptyset$, i.e., $\Psi$ being a morphism on $V$ and $U=\{(\Psi(x), x)\}_{x\in V}$.
In the remaining part, with $\tilde Y_t$ in place of $\mathbb{P}^N_t$ in Lemma \ref{proper},
by exactly the same arguments one can show that
any $(y, x)\in \hat p_{\mathcal{X}}^{-1}(x)\subset \widehat{\mathcal{G}(\Psi)}$ can be approached by
a sequence $(y_j, x_j)\in U$ with $x_j\in V$ and $y_j=\Psi(x_j)$, in such a way that 
$y=\lim_j y_j\in \lim_j\tilde{Y}_{t_j}\subset\tilde{Y}_t$.
It implies the similar conclusion $\Psi(X_t)\subset \tilde Y_t$ for every $t\in \Delta$.
As in Lemma \ref{proper},  corresponding to every $t_k$ given above, there is a $(y_k, x_k)\in \hat p_{\tilde{\mathcal{Y}}}^{-1}(W_3)\subset \widehat{\mathcal{G}(\Psi)}$ 
with $t_k=\pi(\hat p_{\mathcal{X}}(y_k, x_k))=\pi(x_k)$, i.e., $x_k\in X_{t_k}$, 
such that $\pi_{\tilde{\mathcal{Y}}}(W_3)\ni\pi_{\tilde{\mathcal{Y}}}(y_k)=t_k\to \partial\Delta$, contradicting that $W_3$ is compact.
\end{proof}

As promised, let's treat the smoothness issue above before moving forwards.
\begin{step}\label{step 3}
Smoothness of $\mathcal{Y}$ at $\Phi(x_b)$
\end{step}
Since $x_b\in X_b$, the idea is to refine the process of choosing
$b\in\Delta$ in Step \ref{step 1} in such a way that $\Phi(X_b)$ is not entirely contained in the set  of
singular points of $\mathcal{Y}$.  If so, it follows by an analogous procedure as before that, one
can choose $x_b\in X_b$ such that $\mathcal{Y}$ is smooth at $\Phi(x_b)$, as desired.    This refinement is as follows.

Again, we start with
a global line bundle $L$ on $\mathcal{X}$ which is $\pi$-big; it follows that
 $H^0(X_t, L^{\otimes q(t)}|_{X_t})$ gives a bimeromorphic embedding of $X_t$, for every $t\in \Delta$ together with
a choice of $q(t)\in \mathbb{N}$ that depends on $t$.  Note that if
$$\mathcal{E}:=\{e_1, e_2, \cdots, e_k\}\subset H^0(X_t, L^{\otimes q(t)}|_{X_t})$$ contains a basis,
then the Kodaira map associated with $\mathcal{E}$ is still a bimeromorphic map on $X_t$.
The set $\Delta$ is uncountable while $\mathbb{N}$ is countable;
we easily infer that there exists an uncountable set $\Lambda\subset\Delta$ and some $q\in \mathbb{N}$ such that
$q(t)=q$ for each $t\in \Lambda$.    With this given $q$, $h^0(X_t, L^{\otimes q}|_{X_t})$ is locally constant
outside a proper analytic set $W_4\subset \Delta$ as seen by using Theorem \ref{Upper semi-continuity} or
by \cite[3) of Theorem 1.4, p. 6]{Ue} which, via \eqref{gbciv} of Lemma \ref{gbc} and Corollary \ref{exact-cri},
gives in Corollary \ref{gct} the cohomological flatness of $L^{\otimes q}$ in dimension $0$ over $\Delta\setminus W_4$.
By reducing $\Lambda$ while maintaining uncountability,
we can assume that $\Lambda\subset\Delta\setminus W_4$ since $W_4$ in this case can
only be a discrete subset of $\Delta$.
We can further reduce $\Lambda$ and assume that $\Lambda$ is relatively compact in $\Delta\setminus W_4$
(so that $\bar\Lambda\subset \Delta\setminus W_4$) since with $\Lambda=\cup_i(\Lambda\cap O_i)$ for a covering
of $\Delta\setminus W_4$ by countably many  open and relatively compact subsets $O_i$, $\Lambda\cap O_{i_0}$
must be uncountable for some $i_0$.   Also, one sees that there exists a $z\in \Delta\setminus W_4$ such that
for any neighborhood $\mathcal{U}\ni z$, $\mathcal{U}\cap \Lambda$ remains uncountable.  For, by a similar argument as above working
on $O_{i_0}$ there exist $$O_{i_0}=O'_{1}\supset O'_{2}\supset O'_{3}\supset\cdots$$
such that $\Lambda\cap O'_{k}$ is
uncountable for each $k=1, 2, \cdots,$ and $O'_{k}$ converges to some $z\in \bar O_{i_0}\subset \Delta\setminus W_4$.

Having the above uncountable subset $\Lambda$, we are ready to reconstruct the Kodaira
map $\Phi$.  As in Step \ref{step 1}, by Theorem A of Cartan (= Theorem \ref{cartan-a}) we can choose a set $\mathcal{T}$ linearly spanned by sections
$$\{s_0,s_1,\cdots,s_N\}\subset \pi_*L^{\otimes q}(\Delta)$$
such that the germs $(s_0)_z, (s_1)_z,\cdots,(s_N)_z$ at $z$ generate the stalk $(\pi_*L^{\otimes q})_z$,
hence that they generate $(\pi_*L^{\otimes q})_w$ for every $w$ in a neighborhood $U$ of $z$ as a property from
the coherent sheaves (cf. \cite[Lemma 7.1.3]{h90}).  We denote by $\Phi_\mathcal{T}$ the Kodaira map associated with
$\mathcal{T}$.  As said, $\Lambda\cap \mathcal{N}$ is uncountable for any neighborhood $\mathcal{N}\ni z$, so $\Lambda_U:=\Lambda\cap U$ remains uncountable.
By the cohomological flatness above, the natural map
$$\lambda:=\lambda_w: (\pi_*L^{\otimes q})_w\to H^0(X_w, L^{\otimes q}|_{X_w})$$
is surjective for every $w\in U\setminus W_4\neq\emptyset$.
This, together with our preceding construction,
yields that $\Phi_\mathcal{T}$, when restricted to $X_t$,  induces a bimeromorphic embedding of $X_t$ for every
$t\in \Lambda_{U}\subset U\setminus W_4$.  Here, the images
$\{\lambda(s_0), \lambda(s_1), \cdots,\lambda(s_N)\}\subset H^0(X_t, L^{\otimes q}|_{X_t})$ may
not be linearly independent, nevertheless they give a bimeromorphic embedding on $X_t$ as already remarked.

We denote by $\mathcal{Y}_\mathcal{T}$ the image  $\Phi_\mathcal{T}(\mathcal{X})$.
Write $$\mathcal{S}=\mathcal{S}_h\cup \mathcal{S}_v$$ for the set of singular points of
$\mathcal{Y}_\mathcal{T}$, where the vertical part $ \mathcal{S}_v$ consists of those (irreducible) components of $\mathcal{S}$ that are mapped to
points in $\Delta$ via $p_{\mathcal{Y}_\mathcal{T}}:\mathcal{Y}_\mathcal{T}\to\Delta$.   Set
$$\Xi:=p_{\mathcal{Y}_\mathcal{T}}(\mathcal{S}_v)\subset\Delta.$$
There are only countably many components
in $\mathcal{S}_v$, say, by Remmert's proper mapping theorem (= Theorem \ref{remmert}) so that $\Xi$ is discrete, and by that every fiber $X_t$ is compact.  The components in $\mathcal{S}_v$ are fiberwise separated away from one another as $\Xi$ is discrete.
 Since  $\Xi$ is countable and $\Lambda_U$ is uncountable (or $\Xi$ is discrete and
$\Lambda_U$ is not discrete), $\Lambda_U\setminus \Xi\neq\emptyset$.  Now pick some $o\in\Lambda_U\setminus \Xi$.
This means that the $\Phi_\mathcal{T}$-image of $X_o$ is not contained in $\mathcal{S}_v$ and in turn,
neither in singularities $\mathcal{S}$ of $\mathcal{Y}_\mathcal{T}$.
The smoothness in the beginning is proved, so long as with the choice of sections $E$ in Step \ref{step 1}
changed to $\mathcal{T}$, the $X_b$, $\Phi$ ($=\Phi_E$), $\mathcal{Y}$ ($=\mathcal{Y}_E$) etc. in Step \ref{step 1}
are replaced by $X_o$, $\Phi_\mathcal{T}$, $\mathcal{Y}_\mathcal{T}$ throughout.  Of course, with these $X_o$, $\Phi_\mathcal{T}$, $\mathcal{Y}_\mathcal{T}$ the reasoning for Step \ref{step 2} remains unaltered.  This completes our Step \ref{step 3}.

We are left with the bimeromorphic problem of $\Phi|_{X_t}$. Henceforth, we drop the subscript $\mathcal{T}$ in
$\Phi_\mathcal{T}$, $\mathcal{Y}_\mathcal{T}$, etc.
\begin{step}\label{step 4}
  Bimeromorphic embedding of $\Phi|_{X_t}$
\end{step}
Since $\Phi$ is bimeromorphic as just shown, $\Phi$ induces a biholomorphism $$\hat\Phi:\mathcal{X}\setminus \mathcal{S}_1\cong\mathcal{Y}\setminus \mathcal{S}_2$$
with some proper subvarieties $\mathcal{S}_1$, $\mathcal{S}_2$
of $\mathcal{X}$, $\mathcal{Y}$, respectively. 
Writing $\mathcal{S}_1'\subset \mathcal{S}_1$ for the union of vertical divisors, namely the divisors in $\mathcal{S}_1$ which are mapped to
a point of $\Delta$, and $\Delta_{\delta}$ for $\pi(\mathcal{S}_1')$, we see that $\Delta_{\delta}$ is a proper analytic subset of $\Delta$
by the proper mapping theorem of Remmert (= Theorem \ref{remmert}) since $\mathcal{S}_1'\ne\mathcal{X}$.
For every $t\in \Delta\setminus \Delta_{\delta}$, $X_t\setminus \mathcal{S}_1$ is thus a nontrivial open subset of $X_t$ by
dimension reason, and
$\Phi|_{X_t\setminus \mathcal{S}_1}=\hat\Phi|_{X_t\setminus \mathcal{S}_1}$ is a biholomorphism in view that  $\hat\Phi$ is so.
As $\Phi|_{X_t}:X_t\dashrightarrow Y_t$ exists as a meromorphic map for every $t\in\Delta$ as shown in Step \ref{step 1},
we conclude the following.  In an easier way similar to the bimeromorphism of $\Psi$ in Step \ref{step 2} due to the compactness of $X_t$ and $Y_t$, $\Phi|_{X_t}:X_t\dashrightarrow \Phi|_{X_t}(X_t)$ is a bimeromorphism
for every $t\in \Delta\setminus \Delta_{\delta}$, and
$\Phi|_{X_t}:X_t\dashrightarrow Y_t$ is a bimeromorphic map for any $t\in \Delta\setminus \Delta_{\delta}$
such that $Y_t$ is irreducible (see \cite[pp. 17, 13]{Ue}).
In the following we shall work on this irreducibility problem of $Y_t$.

Let's first see some connections of the irreducibility of $Y_t$ with the indeterminacies  $\mathcal{S}(\Phi)$ of $\Phi$.
Write
$$Y_t=\Phi|_{X_t}(X_t)\cup Y_t'$$
where $\Phi|_{X_t}$
is meromorphic on $X_t$ as seen in Step \ref{step 1}, and $Y_t'$ is the union of the remaining components in $Y_t$.
A word of caution is in order.  That is,
$\Phi(X_t)$ and $\Phi|_{X_t}(X_t)$ may be different (\cite[p. 17]{Ue}).  We shall now
see that $Y_t'\subset\Phi(\mathcal{S}(\Phi))$.
Suppose that there exists $p_t\in Y_t'\setminus \Phi|_{X_t}(X_t)$ such that
$p_t=\Phi(q_{t'})$ for some $q_{t'}\in\mathcal{X}$ at which $\Phi$ is a morphism.
Then $t'\neq t$.  For, if $t'=t$, then $\Phi|_{X_t}$ is also a morphism at $q_{t'}=q_t$
so that $p_t=\Phi(q_{t})=\Phi|_{X_t}(q_{t}) \in \Phi|_{X_t}(X_t)$, contradicting $p_t\not\in \Phi|_{X_t}(X_t)$.  But if $t'\neq t$ and
$\Phi$ is a morphism at $q_{t'}$, then $p_t=\Phi(q_{t'})\in Y_{t'}$, again a contradiction
since $Y_t\cap Y_{t'}=\emptyset$ if $t\ne t'$.   To sum up,
it is now $$(Y_t'\setminus \Phi|_{X_t}(X_t))\cap \Phi(\mathcal{X}\setminus
\mathcal{S}(\Phi))=\emptyset$$ so that we must have $Y_t'\subset \Phi(\mathcal{S}(\Phi))$ since $\Phi(\mathcal{X})=\mathcal{Y}$. 
In the case where $Y_{t}'=\emptyset$,
$$Y_t=\Phi|_{X_t}(X_t)\cup Y_t'=\Phi|_{X_t}(X_t)$$ so that  $Y_t$ is irreducible since $X_t$ is so.  The latter assertion
is standard.  For, the graph $\mathcal{G}(\Phi|_{X_t})$ is irreducible, so by the above that $Y_t=\Phi|_{X_t}(X_t)$,
equivalently, that the projection $\mathcal{G}(\Phi|_{X_t})\to Y_t$ is surjective (\cite[p. 16]{Ue}), $Y_t$ is irreducible.

It seems of interest to decide what parts of $\mathcal{S}(\Phi)\subset\mathcal{X}$ are carried by $\Phi$ to $Y_t'$ above.
We hope to discuss this elsewhere.  In what follows, we take another approach to the above irreducibility problem
of $Y_t$.

For use shortly, some preparations are in order.
Let $\tilde{\mathcal{G}}$ be a desingularization of the graph $\mathcal{G}(\Phi)=\mathcal{G}$ with $\tilde {\mathcal{G}}\to \mathcal{G}$ being
proper (cf. the beginning remarks in Step \ref{step 2}),
and the projections $${\tilde p}_{\mathcal{X}},\quad  {\tilde p}_{\mathcal{Y}}$$
the composites of morphisms $\tilde {\mathcal{G}}\to {\mathcal{G}}\to \mathcal{X}$, $\tilde {\mathcal{G}}\to {\mathcal{G}}\to \mathcal{Y}$, respectively. 
Set
$$\pi:\mathcal{X}\to \Delta\quad \text{and}\quad \pi_\mathcal{Y}:\mathcal{Y}\to \Delta$$
as the standard projections.
There exists a proper analytic subset $\Delta_{s}\subset\Delta$ such that the fiber $\tilde {\mathcal{G}}_t:=(\pi\circ
{\tilde p}_{\mathcal{X}})^{-1}(t)$
is smooth for every $t\not\in \Delta_{s}$ (\cite[p. 8]{Ue}).

We claim that $\tilde {\mathcal{G}}_t$, hence ${\mathcal{G}}_t$, is connected for every $t\in \Delta$.
By construction ${\tilde p}_{\mathcal{X}}:\tilde {\mathcal{G}}\to \mathcal{X}$ is a proper modification (\cite[pp. 22, 13]{Ue}) so that
$ ({\tilde p_\mathcal{X}})_*\mathcal{O}_{\tilde{\mathcal{G}}}=\mathcal{O}_{\mathcal{X}}$ (\cite[Corollary 1.14]{Ue}).  We arrive at
${\pi}_*(({\tilde p_\mathcal{X}})_*\mathcal{O}_{\tilde{\mathcal{G}}})=\mathcal{O}_\Delta$ since every fiber
of $\pi:\mathcal{X}\to \Delta$ is connected so that ${\pi}_*\mathcal{O}_{\mathcal{X}}=\mathcal{O}_\Delta$
(\cite[Proposition 1.13]{Ue}).
The fact that $({\pi}\circ {\tilde p_{\mathcal{X}}})_*\mathcal{O}_{\tilde{\mathcal{G}}}=\mathcal{O}_\Delta$ allows us to
conclude that $\tilde {\mathcal{G}}_t$ is connected for every $t\in \Delta$.  This assertion in the algebraic setting,  is standard
(cf. \cite[Corollary 11.3]{Ht}); in the analytic setting here, one can find a proof in
\cite[the paragraph above Corollary 2.13 on p. 111]{bs}.
With the connectedness of $\tilde {\mathcal{G}}_t$, we are going to show that $\Phi(X_t)=Y_t$ for every $t\in \Delta$,
as claimed in the second half of Step \ref{step 1}.

Our proof of $\Phi(X_t)=Y_t$ uses again the desingularization $\tilde{\mathcal{G}}$.  First note the following.
Let $\varphi:U\dashrightarrow V$ be a meromorphic map
between complex spaces.  Assume that $q_1:U_1\to U$ and $q_2:U_2\to U_1$ are proper modifications such that $h_1:=\varphi\circ q_1$ and $h_2:=\varphi\circ q_1\circ q_2$ are morphisms.  Let $F\subset U$ be a subset and write
$B_1=h_1(q_1^{-1}(F))$, $B_2=h_2((q_1\circ q_2)^{-1}(F))$ as subsets of $V$.   Then $B_1=B_2$.    For, $h_2=h_1\circ q_2$ so that
$$B_2=(h_1\circ q_2)((q_1\circ q_2)^{-1}(F))=h_1(q_1^{-1}(F))=B_1$$
as claimed.

Due to this, the original definition of $\Phi(X_t)$ using $\mathcal{G}$ (\cite[p. 14]{Ue}) can now be rewritten as
$\Phi(X_t)={\tilde p}_{\mathcal{Y}}({\tilde p}_{\mathcal{X}}^{-1}(X_t))$ by using $\tilde{\mathcal{G}}$ if one sets
$F=X_t$ and $U=\mathcal{X}, V=\mathcal{Y}, \varphi=\Phi, U_1=\mathcal{G}, U_2=\tilde{\mathcal{G}}$ in the above.
Now one sees that
$\Phi(X_t)={\tilde p}_{\mathcal{Y}}({\tilde p}_{\mathcal{X}}^{-1}(X_t))$  is a compact analytic subset of $\mathcal{Y}$.
For,  ${\tilde p}_{\mathcal{X}}$ is proper so that ${\tilde p}_{\mathcal{X}}^{-1}(X_t)$ is compact,
hence ${\tilde p}_{\mathcal{Y}}({\tilde p}_{\mathcal{X}}^{-1}(X_t))$ is compact.
Further, the restriction of ${\tilde p}_{\mathcal{Y}}$ to the compact subset
${\tilde p}_{\mathcal{X}}^{-1}(X_t)$ is necessarily proper,
regardless of whether the entire ${\tilde p}_{\mathcal{Y}}$ is proper.
(The entire projection ${\tilde p}_{\mathcal{Y}}:\tilde {\mathcal{G}}\to \mathcal{Y}$ is actually proper by the arguments similar to
Lemma \ref{proper} in Step \ref{step 1}; see Remark \ref{proper-rem}.  But we do not need this here.)
By Remmert's proper mapping theorem (= Theorem \ref{remmert}) applied
to ${\tilde p}_{\mathcal{Y}}|_{{\tilde p}_{\mathcal{X}}^{-1}(X_t)}$, its image is an analytic set.
Hence $\Phi(X_t)$ is analytic and compact, as claimed.

As $\pi_\mathcal{Y}:\mathcal{Y}\to\Delta$ is proper, $\pi_\mathcal{Y}(\Phi(X_t))$ is also a compact analytic set in $\Delta$.
Observe that ${\tilde p}_{\mathcal{X}}^{-1}(X_t)=\tilde {\mathcal{G}}_t$ so that $\Phi(X_t)=
{\tilde p}_{\mathcal{Y}}(\tilde {\mathcal{G}}_t)$, and ${\tilde p}_{\mathcal{Y}}(\tilde {\mathcal{G}}_t)$ is thus connected since
 $\tilde {\mathcal{G}}_t$ is so for every $t\in \Delta$, as given in the preceding paragraph.
In short, $\pi_\mathcal{Y}(\Phi(X_t))$ is a connected compact analytic set in $\Delta$, thus it can only be a single point
$\theta\in \Delta$.   As noted in Step \ref{step 1}, $X_t\not\subset \mathcal{S}(\Phi)$ so that $\pi_\mathcal{Y}(\Phi(x))=t$ for $x\in X_t$
at which $\Phi$ is a morphism.  We conclude that the single point $\pi_\mathcal{Y}(\Phi(X_t))=\theta$ equals $t$, equivalently
$$\Phi(X_t)\subset Y_t\quad \text{for every $t\in \Delta$.}$$
By combining $\Phi(\mathcal{X})=\mathcal{Y}$ (\cite[pp. 14, 16]{Ue})
and $\Phi(\mathcal{X})=\cup_{t\in \Delta}\Phi(X_t)$ as seen from their definitions (\cite[p. 14]{Ue}), one sees that the above
inclusion $``\subset"$ is necessarily an equality, that is
$$\Phi(X_t)=Y_t,\ \text{for every $t\in \Delta$},$$
 as to be proved.
However, as it was hinted in the second half of Step \ref{step 1}, this is not equivalent to claiming that $\Phi|_{X_t}(X_t)=Y_t$ for every $t\in \Delta$.

Remark that the above proof does not use the
fact that $\Phi$ is generically one to one or $\hat \Phi$ is biholomorphic; it is valid generally
but we omit the precise formulation.

We are ready to finish the proof for the bimeromorphic property of $\Phi|_{X_t}$.
Recall that in the preceding paragraphs, $\tilde {\mathcal{G}}_t$ is smooth for every $t\in \Delta\setminus \Delta_{s}$ and
is connected for every $t\in \Delta$.  In particular, $\tilde {\mathcal{G}}_t$ is irreducible for every $t$ outside $\Delta_{s}$.
Recalling the above that $\Phi(X_t)=Y_t$ for every $t\in \Delta$ and that $\Phi(X_t)={\tilde p}_{\mathcal{Y}}(\tilde {\mathcal{G}}_t)$,
which is now seen to be irreducible since $\tilde {\mathcal{G}}_t$ is (for $t\not\in \Delta_{s}$),
one obtains that $Y_t$ is irreducible for $t\in\Delta\setminus \Delta_{s}$.
Also recall in the beginning of this step
that for every $t\in \Delta\setminus \Delta_{\delta}$, $\Phi|_{X_t}:X_t\dashrightarrow \Phi|_{X_t}(X_t)$ is a
bimeromorphism, and that $\Phi|_{X_t}:X_t\dashrightarrow Y_t$ is also bimeromorphic provided that $Y_t$ is irreducible.
We have now established that
 $$\Phi|_{X_t}:X_t\dashrightarrow Y_t$$
 is a bimeromorphism
for every $t\in \Delta\setminus (\Delta_{\delta}\cup \Delta_{s})$ where $\Delta_{\delta}\cup \Delta_{s}$ is an analytic subset of $\Delta$.
This completes Step \ref{step 4} and also the proof for the second half of Theorem \ref{thm-gauduchon-update}.
\end{proof}

\begin{rem}\label{rem4.23}
The key property of $\Delta$ that is used in the bimeromorphic embedding of $\mathcal{X}\to \Delta$ above
is its Steinness, so that Theorem A of Cartan (= Theorem \ref{cartan-a}) is allowed.  At this point,  the same conclusion holds if $\Delta$ is replaced
by more general spaces (cf. the proof of Theorem \ref{thm-gauduchon-update'} for use).
We omit the details of the precise formulation.
\end{rem}

If the strongly Gauduchon condition is assumed solely on the central fiber $X_0$, one has the following
variant of the above theorem.

\begin{thm}\label{thm-gauduchon-update'}
Let $B$ be an uncountable subset of $\Delta$. Suppose that the fiber $X_t:=\pi^{-1}(t)$ is Moishezon for each $t \in B$ and the reference fiber $X_0$ admits a strongly Gauduchon metric as in Definition \ref{sGau}.  Then there exists a small constant $\epsilon >0$ such that $X_t$ is Moishezon for each $t\in\Delta_{\epsilon}:=\{z\in \mathbb{C}: |t|<\epsilon\}$. In particular, $X_0$ is Moishezon.
Moreover, there exists an open subset $U\subset \Delta$ with $\Delta\setminus U$ a discrete subset of $\Delta$ such that
the statements analogous to \eqref{thm-gauduchon-update-ii} of Theorem \ref{thm-gauduchon-update} with $\Delta$ replaced by $\Delta_{\epsilon}\cup U,$ hold.
\end{thm}
\begin{proof}
Lemma \ref{sdsGau} implies that there exists a small disk $\Delta_\epsilon\subset \Delta$ of $t=0$, such that $X_t$ is strongly Gauduchon for each $t\in \Delta_\epsilon$.   Since $B$ is not {\it a priori} known to be contained in $\Delta_{\epsilon}$,
we cannot directly apply the previous arguments to $\Delta_\epsilon$.   Instead,
with notations in {\it proof} of Theorem \ref{thm-gauduchon-update}, let $U$ be the open subset of $\Delta$
in the beginning of that proof such that $L$ is $\pi$-big on $\mathcal{X}_U$.   By equipping the subfamily
$\mathcal{X}_{U\cap\Delta_\epsilon}\subset\mathcal{X}_{\Delta_\epsilon}$ with this line bundle $L$
we can now safely apply the previous arguments to $\mathcal{X}_{\Delta_\epsilon}$ by using $L$ and the
strongly Gauduchon conditions within $\Delta_{\epsilon}$.  We obtain
the first part of the theorem.
For the second part, $\Delta'=U\cup\Delta_\epsilon$ is an open subset of $\Delta$ and is still connected.
We apply the preceding proof of Theorem \ref{thm-gauduchon-update} to $\Delta'$ in place of $\Delta$.
Then the second part of the theorem follows; see Remark \ref{rem4.23}. 
\end{proof}

\subsection{Examples for Theorem \ref{thm-moishezon}}
The goal of this subsection is to establish  examples for Theorem \ref{thm-moishezon}.

We give a brief review of Siu--Demailly's solution of \emph{Grauert--Riemenschneider conjecture}: If a compact complex manifold possesses a Hermitian holomorphic line bundle whose curvature is semi-positive everywhere and strictly positive at one point of the manifold, then this manifold is Moishezon.
\begin{defn}[{}]\label{semipos} {A compact complex manifold is called \emph{semi-positive Moishezon}} if there exists a Hermitian holomorphic line bundle on this manifold, whose curvature is semi-positive everywhere and strictly positive at one point. By Siu's criterion \cite{s84}, this manifold is Moishezon.
\end{defn}

Let $E$ be a holomorphic vector bundle of rank $r$ and $L$ a holomorphic line bundle on a compact complex manifold $X$ of dimension $n$. If $L$ is equipped with a smooth Hermitian metric $h$ of Chern curvature form $\Theta_{L,h}$, we define the \emph{$q$-index set} of $L$ to be the open subset
$$X(L,h,q)=\left\{x\in X: \text{$\sqrt{-1}\Theta_{L,h}$ has $q$ negative eigenvalues and $n-q$ positive eigenvalues}\right\}$$
for $0\leq q\leq n$. We also introduce
$$X(L,h,\leq q)=\bigcup_{0\leq j\leq q}X(L,h,j).$$
\begin{thm}[{\cite{Dem85}}]\label{hmi}
With the above setting, the cohomology groups $H^q(X,E\otimes L^{\otimes k})$ satisfy the asymptotic inequalities as $k\rightarrow +\infty$:
\begin{enumerate}[$(1)$]
    \item \emph{{(Weak Morse inequality)}} \label{}
$$h^q(X,E\otimes L^{\otimes k})\leq r\frac{k^n}{n!}\int_{X(L,h,q)}(-1)^q\left(\frac{\sqrt{-1}}{2\pi}\Theta_{L,h}\right)^n+o(k^n).$$
    \item \emph{{(Strong Morse inequality)}} \label{}
$$\label{int-cond}
\sum_{0\leq j\leq q}(-1)^{q-j}h^j(X,E\otimes L^{\otimes k})\leq r\frac{k^n}{n!}\int_{X(L,h,\leq q)}(-1)^q\left(\frac{\sqrt{-1}}{2\pi}\Theta_{L,h}\right)^n+o(k^n).
$$
 \end{enumerate}

\end{thm}

Using the strong Morse inequality with $q=1$, Demailly obtained:
\begin{thm}[{\cite{Dem85}}]\label{}
Let $X$ be a compact complex manifold with a Hermitian holomorphic line bundle $(L,h)$ over $X$ satisfying
$$\label{int-cond}
\int_{X(L,h,\leq 1)}\left(\frac{\sqrt{-1}}{2\pi}\Theta_{L,h}\right)^n>0.
$$
Then $L$ is a big line bundle and thus $X$ is a Moishezon manifold.
\end{thm}
Obviously, a semi-positive Moishezon manifold $(X,L,h)$ in the sense of Definition \ref{semipos} satisfies
$$\int_{X(L,h,\leq 1)}\left(\frac{\sqrt{-1}}{2\pi}\Theta_{L,h}\right)^n>0$$
since $X(L,h,1)=\emptyset$ and
$$\int_{X(L,h,0)}\left(\frac{\sqrt{-1}}{2\pi}\Theta_{L,h}\right)^n=\int_{X}\left(\frac{\sqrt{-1}}{2\pi}\Theta_{L,h}\right)^n>0.$$

Under this type of integration conditions and assumptions on fibers $X_t$ for all $t\in \Delta^*$,
the proof for the deformation limit problem can be somewhat simplified:

\begin{thm}\label{thm-moishezon0}
Let the fiber $X_t:=\pi^{-1}(t)$ be Moishezon for each $t\in \Delta^*$ and admit a Hermitian holomorphic line bundle $(L_t,h_t)$ satisfying Demailly's integration condition
$$\label{0int-cond}
\int_{X(L_t,h_t,\leq 1)}\left(\frac{\sqrt{-1}}{2\pi}\Theta_{L_t,h_t}\right)^n>0.
$$
Suppose that the reference fiber $X_0$ satisfies the local deformation invariance for Hodge number of type $(0,1)$ or admits a strongly Gauduchon metric as in Definition \ref{sGau}. Then $X_0$ is still Moishezon.
\end{thm}
\begin{proof} We deal with the Hodge number case first.
Recall that any Moishezon manifold satisfies the $\partial\bar\partial$-lemma by  \cite{par} or \cite[Theorem 5.22]{DGMS} and thus follows the degeneracy of Fr\"olicher spectral sequence at $E_1$. So it satisfies the deformation invariance of all-type Hodge numbers by \cite[Proposition 9.20]{V} or also \cite[Theorem 1.3]{RZ15}. By assumption, Grauert's continuity theorem \cite[Theorem 4.12.(ii) of Chapter III]{bs} (or just Lemma \ref{gct} above) implies that $R^2\pi_*\mathcal{O}_{\mathcal{X}}$ over $\Delta^*$ is locally free.
Then by using the fact that each of the Moishezon fiber $X_t$ admits a big line bundle $L_t$, the Lebesgue negligibility argument in Subsection \ref{usc} leads to a section $s\in \Gamma(\Delta, R^2\pi_*\mathcal{O}_{\mathcal{X}})$ which arises
from $c_1(L_t)$ and proves to be satisfying $s|_{\Delta^*}=0$, and thus $s=0$ by combining Propositions \ref{Hodge-torsionfree} and \ref{s-torsion}.
So there exists a holomorphic line bundle $L$ on $\mathcal{X}$ such that for some $t_0\in \Delta^*$, the Hermitian metric  $(L|_{X_{t_0}},\tilde{h}_{t_0})$  satisfies
\begin{equation}\label{dci}
\int_{X(L|_{X_{t_0}},\tilde{h}_{t_0},\leq 1)}\left(\frac{\sqrt{-1}}{2\pi}\Theta_{L|_{X_{t_0}},\tilde{h}_{t_0}}\right)^n>0,
\end{equation}
where the Hermitian metric $\tilde{h}_{t_0}$ is obtained by the $\partial\b\partial$-lemma on $X_{t_0}$ such that $\Theta_{L|_{X_{t_0}},\tilde{h}_{t_0}}=\Theta_{L_{t_0},{h}_{t_0}}$.  Under the deformation invariance of $h^{0,1}$, we can also construct a holomorphic line bundle $L'$ on $\mathcal{X}$ with $L'|_{X_{t_0}}=L_{t_0}$ for this $t_0\in \Delta^*$ as in Remark \ref{get-L0}.   However, for our purpose the equality $c_1(L|_{X_{t_0}})=c_1(L|_{t_0})$ is sufficient as far as \eqref{dci} is concerned.

As for the second case, the argument of Theorem \ref{thm-gauduchon-update} 
with the assumption of strongly Gauduchon metric gives the desired holomorphic line bundle $L$ on $\mathcal{X}$ with the same curvature integration property as \eqref{dci}.  Remark that with this integration condition, one can avoid the use of Proposition \ref{propa}; see below for more.

In summary, one obtains a holomorphic Hermitian line bundle $(L,h)$ on $\mathcal{X}$ and  a hermitian metric on $L_{t_0}:=L|_{X_{t_0}}$ for some $t_0\in \Delta^*$ such that $(L_{t_0},h_{t_0}:=h|_{X_{t_0}})$ satisfies
$$\int_{X(L_{t_0},h_{t_0},\leq 1)}\left(\frac{\sqrt{-1}}{2\pi}\Theta_{L_{t_0},h_{t_0}}\right)^n>0.$$
By Demailly's strong Morse inequality in Theorem \ref{hmi}, one has
$$h^0(X_{t_0},L_{t_0}^{\otimes k})\geq h^0(X_{t_0},L_{t_0}^{\otimes k})-h^1(X_{t_0},L_{t_0}^{\otimes k})\geq \frac{k^n}{n!}\int_{X(L_{t_0},h_{t_0},\leq 1)}\left(\frac{\sqrt{-1}}{2\pi}\Theta_{L_{t_0},h_{t_0}}\right)^n-o(k^n)$$
and thus $L_{t_0}$ is big.

The difficulty here is that we have only one big line bundle $L_{t_0}$ with $t_0\in \Delta^*$ for the moment.  Fortunately, for $|t-t_0|\leq \epsilon$ with some small constant $\epsilon>0$, one still has, by continuity of smooth extension
of the smooth Hermitian metric on $L|_{X_{t_0}}$, that
$$\int_{X(L_{t},h_{t},\leq 1)}\left(\frac{\sqrt{-1}}{2\pi}\Theta_{L_{t},h_{t}}\right)^n>0.$$
By Demailly's strong Morse inequality again, one obtains that $L_t$ is big for $|t-t_0|\leq \epsilon$. So Corollary \ref{unc-big} completes the proof.
\end{proof}

As a direct corollary of Theorem \ref{thm-moishezon0}, one obtains the following result.
\begin{cor}\label{sm02}
If the fiber $X_t:=\pi^{-1}(t)$ for each $t\in \Delta^*$ is semi-positive Moishezon  and the $(0,1)$-Hodge number of $X_0$ satisfies the deformation invariance  or admits a strongly Gauduchon metric as in Definition \ref{sGau}, then $X_0$ is Moishezon.
\end{cor}
\begin{proof}
Here we give a second proof of Corollary \ref{sm02}, which seems of independent interest.

By the proof of Theorem \ref{thm-moishezon0}, there exist  a holomorphic line bundle $L$ on $\mathcal{X}$ and some $\tau\in B$ such that $L_\tau:=L|_{X_\tau}$ is semi-positive on the whole $X_\tau$ and strictly positive at one point of $X_\tau$. The difficulty here is that the line bundle $L_\tau$ is big only at one $\tau$ for the moment. By Berndtsson's solution of Grauert--Riemenschneider conjecture \cite{bn}, there exist $c_0,c_1,\cdots>0$ and some positive integer $N$ such that for all $k>N$, there hold
$$h^q(X_\tau, L_\tau^{\otimes k})<c_q k^{n-q},$$
for all $1\leq q\leq n$ and
$$h^0(X_\tau, L_\tau^{\otimes k})\geq c_0 k^{n}.$$

For any $m> N$ and $1\leq q\leq n$, let
$$V_{m,q}=\{t\in \Delta:h^q(X_t, L_t^{\otimes m})\geq c_q m^{n-q}\}$$ and
$$V_m=\cup_{q=1}^n V_{m,q}.$$ Then $V_m$ is an analytic subset of $\Delta$ but not equal to $\Delta$ since for $m>N$, $t=\tau$ is excluded from $V_m$. So for $m>N$, $V_m$ is a discrete subset of $\Delta$. Now
set $V=\cup_{m>N} V_m,$
which is a countable subset of $\Delta$, and $$\tilde{V}:=\Delta\setminus V$$
which is non-empty and uncountable.
So for $\tilde{\tau}\in \tilde V$,  one has
$$h^q(X_{\tilde{\tau}}, L_{\tilde{\tau}}^{\otimes m})< c_q m^{n-q}$$
for each $1\leq q\leq n$ and $m> N$. Thus by asymptotic Riemann--Roch Theorem \ref{arr} applied to $L_{\tilde{\tau}}^{\otimes m}$, one obtains
$$h^0(X_{\tilde{\tau}}, L_{\tilde{\tau}}^{\otimes m})\geq c_0 m^n$$
for all $m> N$, giving that $L_{\tilde{\tau}}$ is also big on $X_{\tilde{\tau}}$ for each $\tilde\tau\in \tilde V$.
We now apply Corollary \ref{unc-big} to complete the proof.
\end{proof}

\end{document}